\documentclass[a4paper,oneside,11pt]{article}
\usepackage[utf8]{inputenc}    

\usepackage{nicefrac}
\usepackage{amsmath,amssymb,amsthm}
\usepackage{a4wide, dsfont}

\usepackage{array}
\usepackage{booktabs}
\usepackage{tikz}
\usepackage{graphicx}
\usepackage{bbm}   
\usepackage{comment}

\usepackage{url}
\usepackage{nameref}
\usepackage[colorlinks, linkcolor = black, citecolor = black, filecolor = black, urlcolor = blue]{hyperref} 

\usepackage{ifthen}
\usepackage{natbib}
\usepackage{color}
\usepackage{epstopdf}
\usepackage{caption}
\usepackage{diagbox}
\usepackage{subcaption}
\usepackage{arydshln}
\usepackage{longtable}
\usepackage{marvosym}
\usepackage{cases}
\usepackage{enumitem}

\DeclareMathOperator*{\argmin}{argmin}

\newtheorem{satz}{Satz}
\theoremstyle{definition}

\newtheorem{example}{Example}
\newtheorem{remark}[example]{Remark}
\newtheorem*{remark*}{Remark}
\newtheorem{assumption}{Assumption}

\theoremstyle{plain}
\newtheorem{theorem}[satz]{Theorem}
\newtheorem{proposition}[satz]{Proposition}

\newtheorem{lemma}{Lemma}

\newcommand{\mc}{\mathcal}
\newcommand{\dx}{\mathrm{d}}
\newcommand{\dd}{\mathrm{d}}
\newcommand{\e}{\varepsilon}
\newcommand{\bs}[1]{\boldsymbol{#1}} 
\newcommand{\Quad}{\qquad \quad}

\newcommand{\expec}{{\mathbb{E}}}
\newcommand{\prob}{{\mathbb{P}}}

\newcommand{\ind}{\text{\bf{1}}} 

\newcommand{\one}{\mathbbm{1}}

\newcommand{\floor}[1]{\left\lfloor#1\right\rfloor}
\newcommand{\ceil}[1]{\left\lceil#1\right\rceil}
\newcommand{\abs}[1]{\lvert #1 \rvert}
\newcommand{\KL}{\operatorname{KL}}

\newcommand{\card}{\operatorname{card}}

\newcommand{\Oop}[1]{{\operatorname{\mathcal{O}}\left(#1\right)}}
\newcommand{\Oph}[1]{{\operatorname{\mathcal{O}}_{\Prob, h\in ( c/\bs p, h_0]}\left(#1\right)}}

\newcommand{\cO}{\ensuremath{\mathcal{O}}} 

\newcommand{\norm}[1]{{\lVert#1\rVert}} 
\newcommand{\normb}[1]{{\big\lVert#1\big\rVert}}

\newcommand{\absb}[1]{\big|#1\big|}
\newcommand{\Abs}[1]{\Big|#1\Big|}
\newcommand{\abss}[1]{\bigg|#1\bigg|}

\newcommand{\R}{{\mathbb{R}}}
\newcommand{\N}{{\mathbb{N}}}

\newcommand{\Hclass}{\mathcal{H}(\alpha; M)}
\newcommand{\Pclass}{\mc P_T(\beta_0, C_Z)}

\newcommand*{\defeq}{\mathrel{\vcenter{\baselineskip0.5ex \lineskiplimit0pt
			\hbox{\scriptsize.}\hbox{\scriptsize.}}}%
	=}
\newcommand*{\defeql}{ = \mathrel{\vcenter{\baselineskip0.5ex   	\lineskiplimit0pt
			\hbox{\scriptsize.}\hbox{\scriptsize.}}}%
}

\newcommand{\lessneqsim}{\raisebox{-0.15cm}{~\shortstack{$\ll$ \\[-0.05cm]
			$\sim$}}~}
\newcommand{\name}[1]{\textsc{#1}}

\usepackage{marginnote}

\usepackage[colorinlistoftodos,textsize=tiny]{todonotes}
\newcommand{\Comments}{1}
\newcommand{\mynote}[2]{\ifnum\Comments=1\textcolor{#1}{#2}\fi}
\newcommand{\mytodo}[2]{\ifnum\Comments=1%
	\todo[linecolor=#1!80!black,backgroundcolor=#1,bordercolor=#1!80!black]{#2}\fi}

\ifnum\Comments=1               
\fi

\newcommand{\wjb}[2]{w_{\bs j, \bs p}(\bs {#1}; {#2})} 
\newcommand{\wkb}[2]{w_{\bs k, \bs p}(\bs {#1}; {#2})}

\newcommand{\wj}[2]{w_{j}({#1}; {#2})} 

\newcommand{\fkmin}{{f_k}_{\min}}
\newcommand{\fkmax}{{f_k}_{\max}}
\newcommand{\fmin}{{\bs f}_{\min}}
\newcommand{\fmax}{{\bs f}_{\max}}

\newcommand{\Cmax}{C_1}
\newcommand{\Clip}{C_2}
\newcommand{\Ccard}{C_3}
\newcommand{\Csum}{C_4}
\newcommand{\Cink}{C_5}

\newcommand{\leadingzero}[1]{\ifnum #1<10 0\the#1\else\the#1\fi}

\newcommand{\normz}[1]{{\left\lVert#1\right\rVert_2}}

\newcommand{\normi}[1]{{\left\lVert#1\right\rVert_\infty}}
\newcommand{\normib}[1]{{\big\lVert#1\big\rVert_\infty}}

\newcommand{\normiM}[1]{{\left\lVert#1\right\rVert_{\mathrm{M},\infty}}}

\newcommand{\E}{{\mathbb{E}}}
\newcommand{\Prob}{{\mathbb{P}}}

\newcommand{\hclass}{\mathcal{H}(\alpha,L)}

\newcommand{\oY}[1]{\bar{Y}_{#1}}
\newcommand{\oeps}[1]{\bar{\varepsilon}_{#1}}
\newcommand{\oZ}[1]{\bar{Z}_{#1}}

\newcommand{\mule}{\widehat \mu_{n,\bs p, h}}
\newcommand{\mulei}{\widehat \mu_{n,\bs p, h,i}}
\newcommand{\mulestar}{\widehat \mu_{n,\bs p, h^\star}}

\newcommand{\lokpol}{{\widehat{\bs \vartheta}_{n,\bs p,h} }}	
\newcommand{\mulp}{{\widehat{\mu}_{n,\bs p,h}^{\,\mathrm{LP},m}}}

\newcommand{\GG}{\mathcal G}

\begin{document}
	

{\title{From dense to sparse design: Optimal rates under the supremum norm for estimating the mean function in functional data analysis}
\author{Max Berger, Philipp Hermann and Hajo Holzmann\footnote{Corresponding author. Prof.~Dr.~Hajo Holzmann, Department of Mathematics and Computer Science, Philipps-Universität Marburg, Hans-Meerweinstr., 35043 Marburg, Germany} \\
\small{Department of Mathematics and Computer Science}  \\
\small{Philipps-Universität Marburg} \\
\small{\{mberger, herm, holzmann\}@mathematik.uni-marburg.de}}

\maketitle

%
		
%
}

\begin{abstract}
%
We derive optimal rates of convergence in the supremum norm for estimating the Hölder-smooth mean function of a stochastic process  which is repeatedly and discretely observed  with additional errors at fixed, multivariate, synchronous design points, the typical scenario for machine recorded functional data. 
Similarly to the optimal rates in $L_2$ obtained in \citet{cai2011optimal}, for sparse design a discretization term dominates, while in the dense case the parametric $\sqrt n$ rate can be achieved as if the $n$ processes were continuously observed without errors. 
The supremum norm is of practical interest since it corresponds to the visualization of the estimation error, and forms the basis for the construction uniform confidence bands. We show that in contrast to the analysis in $L_2$, there is an intermediate regime between the sparse and dense cases dominated by the contribution of the observation errors. Furthermore, under the  supremum norm interpolation estimators which suffice in $L_2$ turn out to be sub-optimal in the dense setting, which helps to explain their poor empirical performance.  In contrast to previous contributions involving the supremum norm, we discuss optimality even in the multivariate setting, and for dense design obtain the $\sqrt n$ rate of convergence without additional logarithmic factors. We also obtain a central limit theorem in the supremum norm, and provide simulations and real data applications to illustrate our results.   


%
%
%
%

		%
\end{abstract}
\noindent {\itshape Keywords.}\quad asymptotic confidence sets, dense and sparse design, functional data, minimax optimality, supremum norm

\section{Introduction}

Functional data analysis (FDA) refers to situations where the data involve functions or images, potentially as features with additional observations in regression and classification problems.  
The field was popularized by \citet{ramsay1998functional}, and since then has attracted the attention of many theoretical and applied statisticians. 
\citet{wang2016functional}, \citet{morris2015functional} and \citet{cuevas2014partial} contribute overviews of various aspects of the field.

Replication and regularization are two central notions in FDA \citep{morris2015functional, ramsay1998functional}. By replication one makes use of information across the repeated observations of the functions, while regularization takes advantage of information within functions, which in practice are only observed at discrete design points in time or space. 


Fixed, non-random  points which are equal across functions are refered to as a deterministic, synchronous design. This typically arises for machine recorded data, such as temperature and further weather data recorded at regular time intervals at weather stations.
In this paper, for such designs we clarify even in a multivariate setting the phase transition from dense to sparse case, by using minimax optimality over Hölder smoothness classes and the supremum norm as loss function.

The multivariate and in particular the  bivariate settings which we comprehensively deal with in the present paper evidently are of great interest, for example for repeatedly observed images. Much less work is avaible here than for the univariate case,  in which the covariate typically refers to time. A notable exception is \citet{wang2020simultaneous}, who consider images, focus on the dense case and work under a more complex set of assumptions then is considered in the present paper. 

The analysis of estimators in FDA has traditionally been conducted for $L_2$ and pointwise errors \citep{hall2006properties}. Arguably the supremum norm is of more or at least equal practical interest, since it corresponds to the visualization of the estimation error. Further it constitutes the basis for the construction uniform confidence bands, which allow for the assessment of the precision of the estimates. Therefore in  recent years much research effort has been devoted to an analysis of estimation errors in the supremum norm. Notable contributions in the random, asynchronous design are \citet{yao2005functional, li2010uniform, zhang2016sparse}. Rates in the supremum norm and asymptotic confidence bands in the deterministic, synchronous design that we focus on have been investigated by \citet{degras2011simultaneous, chang2017simultaneous, xiao2020asymptotic, kalogridis2022robust}. 

A precise assessment of the performance of methods can however only be given by providing optimality results, which we do in terms of minimax rates under the supremum norm in this paper, and which have not been previously derived in the literature.

More technically speaking, we show that in the sparse regime, a discretization term dominates rather than the estimation error arising from the errors as in the random design case. This is in line with results in \citet{cai2011optimal} for the $L_2$-loss and strong smoothness assumptions on the paths of the processes. In particular, they require that the sample paths have the same smoothness as the target mean function, a much stronger assumption than what is imposed in our work. 
Notably and in contrast to the $L_2$ norm there exists a transition phase between the sparse and dense designs when considering the supremum norm in which the contribution of the errors dominates. Our analysis also shows that simple interpolation estimators without smoothing, which according to  \citet{cai2011optimal} suffice to achieve optimal rates in $L_2$ for dense design, in case of observational errors  actually deteriorate with respect to the supremum norm for a large number of row wise observations. Also note that the rates of convergence presented in \citet{xiao2020asymptotic, kalogridis2022robust} by using penalized spline estimators  are suboptimal. In particular, the parametric $\sqrt n $-rate in the dense design is not attained, which is however very desirable for asymptotic inference.  
The parametric $\sqrt n $-rate in the dense setting allows us to provide a central limit theorem in the supremum norm, which is the basis for the construction of uniform confidence bands, e.g.~by bootstrap methods \citep{degras2011simultaneous, wang2020simultaneous} or by using bounds based on the Kac-Rice formula as recently proposed in \citet{liebl2019fast}. 
As real-data illustrations we consider the biomechanical data set from \cite{liebl2019fast} as well as  daily temperature curves at Nuremberg, Germany. The effect of coarser discretization on the estimation quality is determined by using a confidence band for the estimate based on the finest discretization, and investigating whether estimates with coarser design still lie inside this band.




The paper is organized as follows. In Section \ref{sec:model:estimators} we introduce the model and the linear estimators that we shall consider. We work with generic assumptions on the weights, and in Section \ref{sec:local:pol:estimator} verify them for weights of local polynomial estimates. Section \ref{sec:estimation:rates} contains the main results on rates of estimation of the mean function in the supremum norm. 
Upper bounds and matching lower bounds are presented. We also show convergence in distribution to a Gaussian process. 
%
In Section \ref{sec:sim} we illustrate our theoretical results in a simulation study. In Section \ref{sec:applications} we provide real-data illustrations. 
%
The concluding Section \ref{sec:discuss} describes how to transfer the results to more general designs, and points to open problems for the sparse, asynchronous case.  
 Proofs of the main results are given in Section \ref{sec:main:proofs}, while some technical proofs as well as additional simulations are deferred to an appendix. 

\subsubsection*{Notation}

In the paper we use the following notation. For $ \bs a=(a_1,\ldots,a_d)^\top$ and $ \bs b=(b_1,\ldots,b_d)^\top \in \R^d$ write $\bs a\leq \bs b$ if $a_i\leq b_i$ for all $i=1,\ldots,d$,  $\bs{a}^{\bs{b}} \defeq a_1^{b_1}\cdots a_d^{b_d}$,  $|\bs a|=a_1+\ldots+a_d$,  $\bs a!\defeq a_1!\cdots a_d!$, $ \bs a_{\min} = \min_{1\leq r \leq d} a_r$ and $\partial^{\bs{a}} \defeq \partial_1^{a_1}\ldots\partial_d^{a_d}$ if the coordinates are integers.
Further set $\bs 1\defeq (1,\ldots,1)^\top \in \{1\}^d$. Given $\bs p \defeq (p_1,\ldots,p_d)^\top$ denote by $\{\bs j \in \N^d \mid \bs 1 \le \bs j \le \bs p\}$ the set of $\bs j \in N^d$ that are component wise between $\bs 1$ and $\bs p$, and denote $\sum_{\bs j=\bs 1}^{\bs p}\defeq \sum_{j_1=1}^{p_1} \cdots \sum_{j_d=1}^{p_d}$. If for a vector the indices are written in bold letters this indexes vectors such as $\bs{p_j}\defeq (p_{j_1},\ldots,p_{j_d})^\top$. If only the indices are written bold this indexes a scalar like $\e_{i,\bs j} \defeq \e_{i,j_{1},\ldots,j_{d}}$ or $ Y_{\bs j} \defeq Y_{j_{1},\ldots,j_{d}}$.   
For sequences $(p_n), (q_n)$ tending to infinity, we write $p_n \lesssim q_n$, if $p_n =\Oop{q_n}$, $p_n \simeq q_n$ if $p_n \lesssim q_n$ and $q_n\lesssim p_n$, and $p_n \lessneqsim q_n$, if $p_n = o(q_n)$.\\
The notation $\mc O_{\prob,  h \in (h_1, \bs h_0]}$ stands for stochastic dominance uniformly for $h \in (h_1 , h_0]$. Finally, $\norm{\,\cdot\,}_\infty$ denotes the supremum norm on $[0,1]^d$ (sup-norm in the following).

\section{The model, smoothness classes and linear estimators } \label{sec:model:estimators}

Assume that we have data $(Y_{i, \bs j}, \bs x_{\bs j})$ according to the model
\begin{align}
	Y_{i, \bs j}= \mu(\bs x_{\bs j}) + Z_i(\bs x_{\bs j}) + \e_{i, \bs j} \,, \quad  i=1,\dotsc,n\,, \quad  \bs j =(j_1, \ldots, j_d) \,, \label{eq:model}
\end{align}
where $Y_{i, \bs j}$ are real-valued response variables and $x_{\bs j} \in T \subset \R^d$ the design points. The set $T$ is assumed to be a hypercube and we take $ T=[0,1]^d$ in the following. The mean function $\mu \colon T  \to \R$ is to be estimated, $\e_{i, \bs j}$ are independently distributed errors with mean zero and $Z, Z_1,\dotsc,Z_n$ are centered, independent and identically distributed, square integrable random fields on $T$ distributed as the generic $Z$. We shall assume a Cartesian product  structure $\bs x_{\bs j}=(x_{1,j_1}, \ldots, x_{d,j_d})$, $1 \leq j_k \leq p_k$, $k=1, \ldots, d$, and $x_{k,l} < x_{k,l+1}$,  $1 \leq l \leq p_k - 1$, for the covariates. The number $\bs p^{\bs 1} =\bs p^{\bs 1}(n) =\prod_{k=1}^d p_k(n)$ of design points and the design points $\bs x_{\bs j}$ depend on $n$, which is often suppressed in the notation. 
%
%
%
%

	Given $\alpha >0$ we set $\left\lfloor \alpha \right\rfloor=\max\{  k\in\N_0 \mid k<\alpha \}.$  
	A function $f\colon T \to \R$ is H\"older-smooth with index $\alpha$ if for all indices $\bs \beta=(\beta_1,\ldots,\beta_d)^\top$ with $|\bs \beta|\leq \left\lfloor \alpha \right\rfloor$ the derivatives $\partial^{\bs \beta} f(\bs x)$ exist 
	and the H\"older-norm given by
	$$\norm{f}_{\mc H, \alpha}\defeq\max_{|\bs \beta| \leq \left\lfloor \alpha \right\rfloor } \sup_{\bs x \in T} |\partial^{\bs\beta} f(\bs x)|+ \max_{|\bs \beta|=\left\lfloor \alpha \right\rfloor}\sup_{\bs x,\bs y \in T, \,\bs x\neq \bs y} \frac{|\partial^{\bs \beta} f(\bs x)-\partial^{\bs \beta} f(\bs y)|}{\norm{\bs x-\bs y}_\infty^{\alpha-\left\lfloor \alpha \right\rfloor}}$$
	is finite.  
	The H\"older class on $T$ with parameters $\alpha>0$ and $L>0$ is defined by
	\begin{equation}\label{def:hoelder:class}
		\mc H_T(\alpha, L) = \big\{f\colon T \to \R \mid \norm{f}_{\mc H, \alpha} \leq L \big\}.
	\end{equation}
%
%
%
For the distribution of the processes $Z$ we assume that $\expec[Z(\bs 0)^2] < \infty$ and that there exists a random variable $M = M_Z>0$ with $\expec [M^2] < \infty $ such that 
\begin{equation}\label{eq:hoeldercontpathsZ}
	\big|Z(\bs x)-Z(\bs y)\big| \leq M \, \norm{\bs x-\bs y}_\infty^\beta \quad \text{almost surely}
\end{equation}
for all $\bs x, \bs y \in T$, where $0 < \beta \leq 1$ is a constant.
%
%
Given $C_Z>0$ and $0 < \beta_0 \leq 1$, we consider the class of processes 
\begin{align}
	\Pclass & = \big\{ Z: T \to \R \ \text{random field} \mid \exists \ \beta \in [\beta_0,1] \text{ and } M \text{ s.th.}\nonumber\\
	& \qquad \qquad \expec [M^2] + \expec [Z(\bs 0)^2]\leq C_Z \text{ and \eqref{eq:hoeldercontpathsZ} holds}  \big\}.		\label{eq:classprocesses}
\end{align}	

\begin{remark}
    	Note that Hölder continuity of the process $Z$ in \eqref{eq:hoeldercontpathsZ} is a natural assumption when deriving the parametric $1/\sqrt n$-rate and CLTs in the supremum norm, in view of the Jain Markus Theorem \citep[Example 2.11.13]{van1996weak}. Various contributions which work without direct smootheness assumptions on the paths \citep{li2010uniform, zhang2016sparse, xiao2020asymptotic} only achieve a rate of $(\log (n)/n)^{1/2}$ in the dense setting. While this does not constitute a big difference in terms of rates, for the asymptotics one requires the exact parametric rate. Further, it is of interest when reconstruction can be conducted as if the processes were continuously observed without errors.  
	Let us point out that in contrast to \citet{cai2011optimal}, we however do not assume that the paths of the process $Z$ are as smooth as the mean function. Typically $\mu$ will be in a Hölder class of smoothness $\alpha>0$ which will be larger than $\beta$ in \eqref{eq:hoeldercontpathsZ}. 
\end{remark}

\begin{assumption}[Distributional assumptions] \label{ass:model} 
		The random variables $\{\e_{i,\bs j} \mid 1 \leq i \leq n,\, \bs 1 \le \bs j \le \bs p\}$ are independent and independent of the processes $Z_1,\dotsc,Z_n$. Further we assume that the distribution of $\e_{i,\bs j}$ is sub-Gaussian, and setting $ \sigma_{ij}^2 \defeq  \E[\e_{i,\bs j}^2]$ we have that $\sigma^2 \defeq \sup_n \max_{\bs 1 \leq \bs j \leq \bs p} \sigma_{\bs j}^2 < \infty$ and that there exists $\zeta>0$ such that $\zeta^2\sigma_{i,\bs j}^2$ is an upper bound for the  sub-Gaussian norm of $\e_{i,\bs j}$. \label{ass:errors} 
\end{assumption}

\begin{remark}
	In the case where the distribution of the additional errors $\varepsilon_{i,j}$ is not a sub-Gaussian distribution a Gaussian approximation yields similar rates of convergence under the assumption that $\expec\abs{\varepsilon_{i,j}}^q < \infty$ for some $q \geq 4$ and a additional assumption on the weights of the estimator. The resulting rate of convergence and the proof and the discussion of the dominating rates of convergence for the one-dimensional case $d=1$ can be found in the Appendix \ref{sec:app:gaussian_approximation}.
\end{remark}

%

For the mean function $\mu$, we consider linear estimators of the form
\begin{equation}\label{eq:linest:mu:bar}
		\mule(\bs x) = \sum\limits_{\bs j= \bs 1}^{\bs p}   w_{\bs j,\bs p}(\bs x;h; \bs x_{\bs 1},\dotsc,\bs x_{\bs p}) \, \oY{\bs j} \,, \qquad \text{where} \quad \oY{\bs j} = \frac1n \sum_{i=1}^nY_{i,\bs j},
\end{equation}
%
%
where $\bs x \in T$,  $w_{\bs j,\bs p}(\bs x;h; \bs x_{\bs 1},\dotsc,\bs x_{\bs p})$ are deterministic weights depending on the design points and on a bandwidth parameter $h >0$. 
We often abbreviate the notation by writing $\wjb xh \defeq w_{\bs j,\bs p}(\bs x;h; \bs x_{\bs 1},\dotsc,\bs x_{\bs p})$.
%


%

While the following assumptions on the weights are standard, note that we require them to hold also for $h$ of order $1/\bs p_{\min}$, a regime which is not relevant in traditional nonparametric regression. Here recall the notation $\boldsymbol{p}_{\text{min}} =  \min_{k} p_k$ for $\bs p = (p_1,\ldots,p_d)$.


\begin{assumption}[Weights of linear estimator]\label{ass:weights} 
There are a $c>0$ and a $h_0>0$ such that for  sufficiently large $\bs p_{\min} $,  the following holds for all $h \in (c/\bs p_{\min}, h_0]$ for suitable constants $\Cmax, \Clip>0$ which are independent of $n,\bs p,h$ and $\bs x$.
	\begin{enumerate}[label=\normalfont{(W\arabic*)},leftmargin=9.9mm]
		\item \label{prop_W} The weights reproduce polynomials of a suitable degree $\zeta \geq 1$, that is for $\bs x \in T$,
		\begin{align*} 
			\sum_{\bs 1\leq \bs j \leq \bs p} w_{\bs j, \bs p}(\bs x, h) =1\,, \quad \sum_{\bs 1 \leq \bs j \leq \bs p}(\bs x_{\bs j}-\bs x)^{\bs r} \, w_{\bs j, \bs p}(\bs x;h) = 0\,, \quad \bs r \in \N^d \text{ s.t. } |\bs r| = 1,\dotsc ,\zeta \,.
		\end{align*} \label{ass:weights:polynom}
		\item 
  We have $ w_{\bs j, \bs p}(\bs x;h)=0$ if $ \norm{\bs x_{\bs j}-\bs x}_\infty>h$ with $\bs x \in T$. 
  \label{ass:weights:vanish}
		\item	For the absolute values of the weights, $ \displaystyle \max_{\bs 1\leq \bs j \leq\bs  p} \big|w_{\bs j, \bs p}(\bs x;h)\big|  \leq \frac{\Cmax}{ \bs p^{\bs 1} h^d }$, $\bs x \in T$.  \label{ass:weights:sup}    
		
  %
 %
		%
\item For a constant $\Clip>0$ it holds that
\begin{align*}
	\big|w_{\bs j, \bs p}(\bs x;h) - w_{\bs j, \bs p}(\bs y;h) \big| \leq \frac{\Clip}{\bs p ^{\bs 1} h^d} \bigg(\frac{\norm{\bs x-\bs y}_\infty}h\wedge 1\bigg), \qquad  \bs x,\bs y \in T\,.
\end{align*} \label{ass:weights:lipschitz}
	\end{enumerate}		
\end{assumption}
Moreover, we require the following design assumption.
\begin{assumption}[Design Assumption]\label{ass:design1} 
	
	  There is a constant $\Ccard >0$ such that for each $\bs x \in T$ and $h>0$ we have that 
	$$\card \big\{\bs j \in \{\bs 1,\ldots, \bs p\} \mid \bs x - (h, \ldots, h) \leq \bs{x_j} \leq \bs x + (h, \ldots, h)\big\}  \leq \Ccard\, h^d\,\bs p^{\bs 1}. $$ 
\end{assumption}
Note that for a product design, we could formulate Assumption \ref{ass:design1} coordinatewise. However, in the present formulation it also applies to more general designs. 
In case of design densities, 
 Assumption \ref{ass:design1} is satisfied, and we check the above assumptions on the weights of local polynomial estimators in Section \ref{sec:local:pol:estimator}. 
%
%
%


%


\section{Optimal rates of estimation  for the mean function} \label{sec:estimation:rates}

\subsubsection*{Upper bounds}

%
%


We start by deriving upper bounds on the rate of convergence of the linear estimator in \eqref{eq:linest:mu:bar} under Assumption \ref{ass:model}.
Bounding the terms in the error decomposition
\begin{align}
	\mule(\bs x) - \mu(\bs x) & =  \sum_{\bs k=\bs 1}^{\bs p} \wkb xh\,(\mu(\bs x_{\bs k}) - \mu(\bs x) ) + \sum_{\bs k=\bs 1}^{\bs p} \wkb{x}{h}\, \oeps{\bs k} + \sum_{\bs k=\bs 1}^{\bs p} \wkb{x}{h}\, \oZ{n}(\bs x_{\bs k})\nonumber\\
	& \defeql  I_1^{\bs p,h}(\bs x) + I_2^{n, \bs p,h}(\bs x) + I_3^{n, \bs p,h}(\bs x) \, , \qquad \bs x \in[0,1]^d \,, \label{eq:decomposition}
\end{align}	
where 
\begin{align*}
\oeps{\bs j} = \frac1n \sum_{i=1}^n \varepsilon_{i,\bs j} \qquad \text{and} \qquad \oZ{n}(\bs x) = \frac1n \sum_{i=1}^n Z_i(\bs x) \, ,
\end{align*}
leads to the following result.

\begin{theorem} \label{theorem:estimation:rates} 
In model \eqref{eq:model} under Assumptions \ref{ass:model} and \ref{ass:design1}, consider the linear estimator $\mule$ in \eqref{eq:linest:mu:bar} with weights satisfying Assumption \ref{ass:weights} with $\gamma=\floor \alpha$ for given Hölder smoothness $\alpha>0$. Then for sufficiently large $\bs p_{\min}$ and $n$, 
\begin{equation}\label{eq:upperboundone}
	\sup_{h \in (c/{\bs p_{\min}}, h_0]} \sup_{Z \in \Pclass, \mu \in \Hclass}\, a_{n,\bs p, h}^{-1}\, \E_{\mu,Z} \Big[ \normib{ \mule - \mu} \Big] = \cO(1),
\end{equation}
where 
\begin{equation}\label{eq:orderbound}
a_{n,\bs p, h} = \max \Big(h^\alpha, \Big(\frac{\log (h^{-1})}{n\bs p^{\bs 1} h^d}\Big)^{1/2}, n^{-\frac{1}{2}} \Big).
\end{equation}
Therefore, setting
	\begin{equation}\label{eq:optimal:bandwidth}
	h^\star = \max\Big(c /\bs p_{\min}, \Big(\frac{\log(n\,\bs {p^1})}{n\,\bs{p^1}}\Big)^{\frac1{2\alpha+d}}\Big),
\end{equation}
we obtain
\begin{equation}\label{eq:optimal:rate}
	\sup_{Z \in \Pclass, \mu \in \Hclass}\,  \E_{\mu,Z} \Big[ \normib{ \mulestar - \mu} \Big] = \cO\Big( \max\Big(\bs p_{\min}^{-\alpha} , \Big(\frac{\log(n\,\bs{p^1})}{n\,\bs{p^1}}\Big)^{\frac\alpha{2\alpha+d}} , n^{-\frac{1}{2}} \Big) \Big). 
\end{equation}
\end{theorem}

Here the bound $h^\alpha$ in \eqref{eq:orderbound} arises from the bias term $I_1^{\bs p,h}(\bs x)$ in \eqref{eq:decomposition} by standard bias estimates using property \ref{ass:weights:polynom} of the weights, the bound $n^{-1/2}$ follows from the stochastic process term $I_3^{n, \bs p,h}(\bs x) $ by using $\E[\normib{\oZ{n}}] = \cO(n^{-1/2})$ and boundedness of sums of absolute values of the weights, and the intermediate term in \eqref{eq:orderbound} from $I_2^{n, \bs p,h}(\bs x)$ upon using Dudley's entropy bound \citep[Corollary 2.2.8]{van1996weak}.
We provide the details in Section \ref{sec:proof:theorem:rates}.

\subsubsection*{Lower bounds} \label{ssec:optimality}

Next we show optimality of the upper bounds in Section \ref{sec:estimation:rates} by deriving corresponding lower bounds of estimation in the minimax sense (see \citet[Section 2]{tsybakov2008introduction}).

\begin{theorem}\label{theorem:optimality}
	Assume that in model \eqref{eq:model} the errors $\e_{i, \bs j}$ are i.i.d.~$\mathcal N(0, \sigma_0^2)$ - distributed, $\sigma_0^2 >0$. The design points $\bs{x_1},\ldots,\bs{x_p}$ shall satisfy Assumption \ref{ass:designdensity}. Then setting 
	\[ a_{n,\bs p} = \max\bigg\{ \bs p_{\min}^{-\alpha},\bigg(\frac{\log(n\,\bs{p^1})}{n{\bs p ^1}}\bigg)^{\frac{\alpha}{2\alpha+d}},
	n^{-\frac12}\bigg\} \]
	we have that
	\begin{align*}
				\liminf\limits_{n,\bs p \to \infty} &\inf\limits_{\hat\mu_{n,\bs p}} \sup\limits_{\mu \in\mc 	H(\alpha,L), \, Z \in \Pclass}\E_{\mu, Z} \Big[a_{n,\bs p}^{-1}\;\norm{\hat\mu_{n,\bs p}-\mu}_\infty\Big] \geq c >0,
	\end{align*}
	where the infimum is taken over all estimators $\hat\mu_{n,\bs p}$ of $\mu$. 
\end{theorem}
The proof is given in Section \ref{sec:proof:theorem:optimality}. 
\begin{remark}
We point out that the lower bound does not require the design Assumption \ref{ass:design1}. 
The argument and hence the result can be adapted to more general distributions of the errors. What is required in our argument is that the Kullback-Leibler divergence in the location model associated with the distribution of the errors is quadratic in the location parameter, at least locally around $0$. See for example \citet{tsybakov2008introduction}, Assumption B in Section 2.3 and Section 2.5. Moreover, intuitively speaking Gaussian errors are quite regular, and we show that these already suffice to  achieve a lower bound slow enough to match the upper bound.  The upper bound holds more generally, so slower minimax rates cannot occur, and when moving away from Gaussian errors, one does not expect faster rates	except for degenerate errors. Thus, it appears reasonable to consider Gaussian errors in the lower bounds. 
    
\end{remark}

\subsubsection*{Discussion}

\begin{remark}[Regimes in the rate] \label{rem:regimes_in_the_rate}
 
	Let us discuss the minimax rates given in Theorems \ref{theorem:estimation:rates} and \ref{theorem:optimality} in terms of $\bs p$ and $n$ and relate them to results previously obtained in the literature. 
 For simplicity, we concentrate on the setting  where all coordinates of $\bs p$ are of the same order $p$, which is assumed to be polynomial in $n$.
 
 If $p \lesssim (n/\log n)^{1/(2\, \alpha)}$, the resulting rate in \eqref{eq:optimal:rate} is $p^{-\alpha}$. The discretization term dominates as observed for the error analysis in $L_2$ in \citet{cai2011optimal}. Note that for the asynchronous, random design considered in \citet{li2010uniform, wang2016functional} this term does not arise. 

 If in contrast $p \gtrsim (\log n)^{1/d}\, n^{1/(2\, \alpha)}$  \eqref{eq:optimal:rate} reduces to  $n^{-1/2}$, without the additional logarithmic factors that arise e.g.~in \citet{li2010uniform, wang2016functional, xiao2020asymptotic}.  
 
 In between the rate is $\big(\log(n)/(n\,p^d)\big)^{\alpha/(2\alpha+d)}$, a regime which does not occur in the error analysis in $L_2$ in \citet{cai2011optimal} but has been obtained for the asynchronous, random design in \citet[Corollary 3.2, (a)]{li2010uniform}, see also \citet[Corollary 4.2, (1)]{wang2016functional}.

\end{remark}

\begin{remark}[Choice of the bandwidth]
    Consider the setup of Theorem \ref{theorem:estimation:rates} in the case that $\max\big((\log(p)/\bs{p^1})^{1/d}, \bs{p}_{\min}^{-1} \big) \lessneqsim n^{-\frac1{2\alpha}}$ and define the interval
    \begin{equation*}
        \bar H_n \defeq \big[ c_1\,\max\big(\big(\log(p)/\bs{p^1}\big)^{1/d}, \bs{p}_{\min}\big), \, c_2\,
        n^{-\frac{1}{2\alpha}} \big],
    \end{equation*}
    with constants $c_1, c_2 >0$. Then the result from Theorem \ref{theorem:estimation:rates} specializes to
    \begin{equation}
        \sup_{h \in \bar H_n}\, \sup_{Z \in \Pclass, \mu \in \Hclass}\, \sqrt n\; \expec_{\mu,Z} \big[\norm{ \widehat \mu_{n,p,h} - \mu}_\infty \big]= \cO(1)\,.
    \end{equation}
    Therefore any choice of $h_n\in \bar H_n$ 
    is sufficient in order to achieve the $\sqrt n$-rate. Consider for simplicity the setting where all coordinates in $\bs p$ are of the same order $p$ what reduces the assumption on $p$ to $p \gtrsim (\log n)^{(\zeta + \delta)/d}\, n^{1/(2\, \alpha)}$ for some $\delta \geq 1$ and $\zeta \geq 0$. In this case taking $h \simeq (\log n)^{\delta/d} / p$ already yields the $n^{-1/2}$-rate. If $\zeta$ has to be $0$ then the bias-term is also of order $n^{-1/2}$, and if $\delta = 1$ then the error term is of order $n^{-1/2}$. Therefore asymptotic inference, see Theorem \ref{cor:asymptotic:normality}, is only possible for $\delta>1$ and $\zeta> 0$. Further data-driven bandwidths 
    which are contained in such intervals with probability converging to $1$ are valid in order to obtain the $\sqrt n$-rate.
    
%
\end{remark}

\begin{remark}[Interpolation estimators]\label{rem:errorinterpol}
	Note that the interpolation estimator as suggested in \citet{cai2011optimal} for the dense regime has an error contribution of the form 
	$$\max_{\bs j} |\oeps{\bs j}|$$
	which for independent and identically distributed Gaussian random errors is of order $\sqrt{\log \bs{p^1}}/\sqrt n$. If the coordinates of $\bs p$ are polynomial in $n$ this is of order $\sqrt{\log n}/\sqrt n$, thus slower than $1/\sqrt n$.  
\end{remark}

\section{Asymptotic normality and local polynomial estimators}\label{sec:local:pol:estimator} 

\subsubsection*{Asymptotic normality}

In case the term $I_3^{n,\bs p, h}$ in the decomposition \eqref{eq:decomposition} dominates, we can apply the functional central limit theorem to obtain asymptotic normality of the linear estimator in \eqref{eq:linest:mu:bar}, thereby generalizing \citet[Theorem 1]{degras2011simultaneous} to multiple dimensions and general linear estimators. 
We require the following additional assumption on the process $Z$.


\begin{assumption} \label{ass:model:2} 
		The covariance kernel $\Gamma\colon [0,1]^d\times [0,1]^d \to \R$ of the process $Z\colon \Omega \times [0,1]^d \to \R$ belongs to the \name{Hölder}-class with smoothness $\gamma \in (0,1]$, that is $\Gamma \in \mc H_{[0,1]^{2d}}(\gamma, \tilde L)$.
\end{assumption}

\begin{theorem} \label{cor:asymptotic:normality}
	In model \eqref{eq:model} under Assumptions \ref{ass:model} and \ref{ass:model:2}, consider the linear estimator $\mule$ in \eqref{eq:linest:mu:bar} with weights satisfying Assumption \ref{ass:weights} with $\zeta=\floor \alpha$.  If  for some $\delta >1$ we have that $\log(\bs p_{\min})^{\delta/d}/\big(\min((\bs{p^1})^{1/d}, \bs{p}_{\min} )\big) \lessneqsim n^{-\frac1{2\alpha}}\log(\bs{p}_{\min})^{-\delta} $ ,
    then for all sequences of smoothing parameters $h_n = h$ in 
 $      H_n \defeq \big[c_1\,\log(\bs p_{\min})^{\delta/d}/\big(\min((\bs{p^1})^{1/d}, \bs{p}_{\min} )\big), \, c_2\,
        n^{-\frac{1}{2\alpha}}\log(\bs{p}_{\min})^{-\delta} \big]$ with constants $c_1, c_2>0$,
	we obtain the convergence in distribution
	\begin{align*}
		\sqrt n \big( \mule -\mu\big) \ \stackrel{D}{\longrightarrow} \ \GG (0,\Gamma)\,,
	\end{align*}
	where $\GG$ is a real-valued Gaussian process on $[0,1]^d$ with covariance kernel $\Gamma$ of $Z$. 
\end{theorem}

\begin{remark}
        The assumption $\log(\bs p_{\min})^{\delta}/\big(\min((\bs{p^1})^{1/d}, \bs{p}_{\min} )\big) \lessneqsim n^{-\frac1{2\alpha}}\log(\bs{p}_{\min})^{-\delta} $ for a $\delta >1$ guarantees that the intervals $H_n$ are not empty for $n$ being large enough. The rate of the upper bounds yields
    \begin{align*}
        h^\alpha & \lesssim n^{-1/2}\,\log(\bs{p}_{\min})^{- \alpha\,\delta} \lessneqsim n^{-1/2}\,.
    \end{align*}
    Therefore the discretization error is negligible. Note that instead of $\log(\bs{p}_{\min})^{-\delta}$ any factor that would make up for a strictly faster rate of the discretization error would be permissible. The rate of the lower bound guarantees
    \begin{align*}
        \bigg( \frac{\log(1/h)}{n\,\bs{p^1}h^d}\bigg)^{1/2} & \lesssim \bigg( \frac{\log(\bs{p}_{\min})^{1-\delta}}{n\,\bs{p^1}\min(\bs{p^1}, \bs{p}_{\min}^d)^{-1}}\bigg)^{1/2} \lesssim \bigg( \frac{\log(\bs{p}_{\min})^{1-\delta}}{n}\bigg)^{1/2}\lessneqsim \frac{1}{\sqrt n}\,.
    \end{align*}
    Here $\delta>1$ is required.
\end{remark}

\subsubsection*{Local polynomial estimators}



We discuss local polynomial estimators and show that in case of design densities and mild assumptions on the kernel, the weights satisfy the properties in Assumption \ref{ass:weights}. 

\begin{assumption} \label{ass:designdensity}
	Assume that there exist Lipschitz continuous densities $f_k\colon[0,1]\to \R$ bounded by $0<\fkmin \leq f_k(t) \leq \fkmax < \infty$, $t \in [0,1]$ and $1 \leq k \leq d$, such that the design points $x_{k,l}$,  $1 \leq l \leq p_k $
%
		\begin{equation*}
			\int_0^{x_{k,l}} f_k(t) \, \dx t=\frac{l-0.5}{p_k} \,, \qquad l=1,\ldots,p_k,\quad k=1,\dotsc,d \,.	
		\end{equation*}
Lemma \ref{lemma:design:points} in the Appendix \ref{sec:proof:lemma:locpol:weights} shows that Assumption \ref{ass:designdensity} implies Assumption \ref{ass:design1}.
\end{assumption}

Let $N_{l,d}\defeq \binom{d+l}{d}$ let \(
	\bs \psi_l \colon \{1, \ldots, N_{l,d-1}\} \to \{\bs r \in \{0,1,\ldots, l\}^d \mid |\bs r| = l\} \)
be an enumeration of the set $\{\bs r \in \{0,1,\ldots,l\}^d \mid |\bs r | = l\}$ and let $U_m\colon \R^d \to \R^{N_{m,d}}$ be defined as 
$ U_m(\bs u)\defeq \big(1, P_{\bs \psi_1(1)}(\bs u), \ldots,$ 
$P_{\bs \psi_{1}(d)}(\bs u), \ldots, P_{\bs \psi_m(1)}(\bs u), \ldots, P_{\bs \psi_m(N_{m,d-1})}(\bs u)\big)$
for the monomial $P_{\bs \psi_l(j)}(\bs u)\defeq \frac{\bs u^{\bs \psi_l(j)}}{\bs \psi_l(j)!}$, for all $ j=1,\ldots, N_{l,d-1}$ and $ l= 1, \ldots, m$. As a result of the properties of the binomial coefficient we note that $U_m(\bs u) \in \R^{N_{m,d}}$.
Choose a non-negative kernel function $K \colon \R^d \to [0, \infty)$ and a bandwidth $h >0$, set $K_{h}(\bs x) \defeq K(\bs x/h)$ and $U_h(\bs x) = U_m(\bs x/h)$, and given an order $m \in \N$ define
\begin{align} \label{eq:locpol} 
	\lokpol(x) \defeq {\displaystyle\argmin_{\bs \vartheta \in \R^{N_{m,d}}}} \sum\limits_{\bs j=\bs 1}^{\bs p} \big(\oY{ \bs j} - \bs \vartheta^\top\, U_h(\bs x_{\bs j}-\bs x) \big)^2 K_{h}(\bs x_{\bs j}-\bs x) \,, \quad \bs x \in T \,.
\end{align}
The local polynomial estimator of $\mu$ of order $m$ at $\bs x$ is given as 
\begin{align}\label{eq:locpolest}
	\mulp(\bs x) \defeq \big(\lokpol(\bs x)\big)_1 \,, \quad \bs x \in T \,.
\end{align}
\begin{lemma}\label{lemma:locpol:weights}

Let the kernel function $K$ have compact support in $[-1,1]^d$, be Lipschitz continuous and satisfy 
\begin{align}
	K_{\min}\one_{[-\bs \Delta, \bs \Delta]}(\bs u) \leq K(\bs u) \leq K_{\max} \qquad \bs u \in \R^d \,,\label{eq:kernel}
\end{align}
for positive constants $\Delta, K_{\min}, K_{\max}>0$, where  $\bs \Delta = (\Delta,\ldots, \Delta)^\top \in \R^d$.
	Then under the design Assumption \ref{ass:designdensity},  
there exist $p_0 \in \N$ and $h_0, c>0$ such that for $\bs p = (p_1,\ldots,p_d)^\top$ with $\bs p_{\min} \geq p_0$ and $h\in(c/\bs p_{\min},h_0]$, 
the local polynomial estimator is uniquely defined by \eqref{eq:locpolest}, and is a linear estimator with weights satisfying 
Assumption \ref{ass:weights} with $\gamma = m$.
\end{lemma}
%
 %
%
The proof is given in Section \ref{sec:proof:lemma:locpol:weights} in the appendix.




\section{Simulations} \label{sec:sim}

We illustrate the rates in Theorem \ref{theorem:estimation:rates} in a simulation. We consider dimension $d=1$, and for the mean function choose
\begin{equation}
	\mu_0(x)= \sin(3\pi(2x-1))\exp(-2\abs{2x-1}), \quad x\in [0,1]. \label{eq:sim:mu}
\end{equation}

\begin{figure}[t]
    \centering
		\includegraphics[width=.6\linewidth]{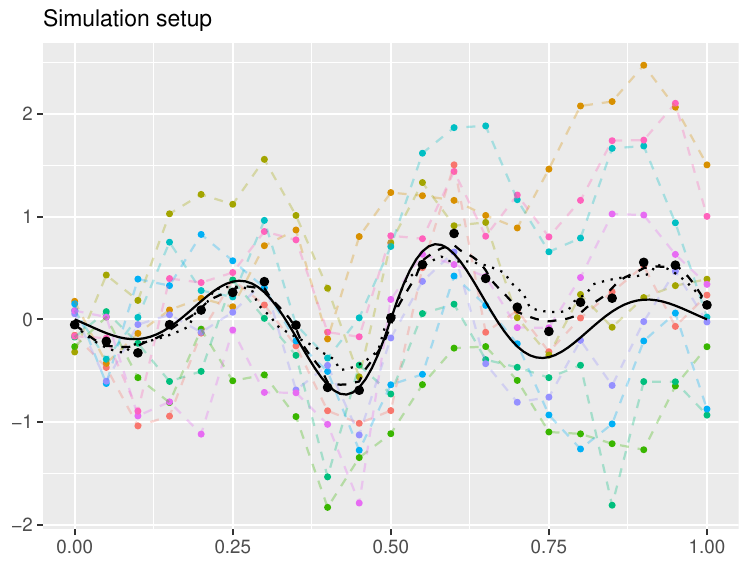}
		\caption{\small Mean function $\mu_0$ (solid line), $n = 10$ observational rows with $p = 21$ design points. Estimator based on all (dashed line) as well as on $p=11$ observations (dotted line).\\\phantom{h}}
		\label{fig:example}
\end{figure}

\begin{figure}[t]
	\begin{minipage}[b]{.45\linewidth}
		\centering
		\includegraphics[width=\linewidth]{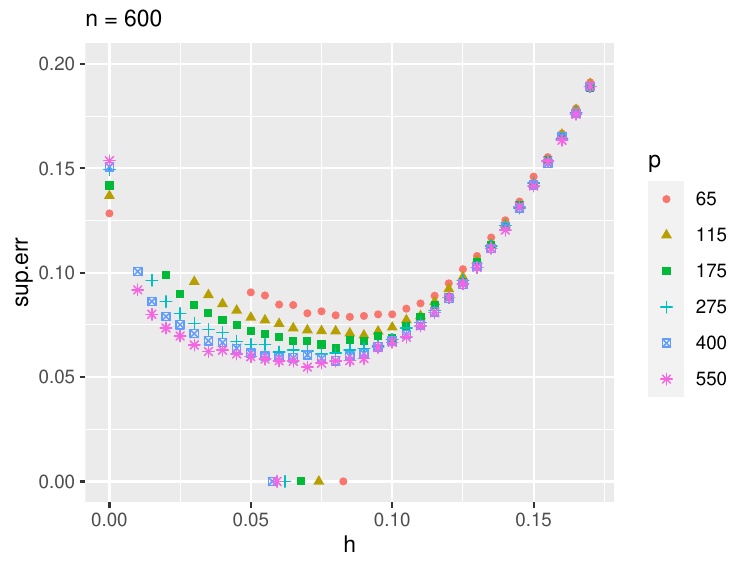}
		\caption{\small Sup-norm error for $n=600$ and various sizes of $p$ for bandwidths on a $0.005$ grid starting at $3/p$, as well as error of interpolation estimator. The values with \texttt{sup.err = 0} indicate the bandwidth chosen by leave-one-curve-out cross validation.}
		\label{fig:comp:h:sigma1}
	\end{minipage}
\hspace{0.5cm}
	\begin{minipage}[b]{.45\linewidth}
		\centering
		\includegraphics[width=\linewidth]{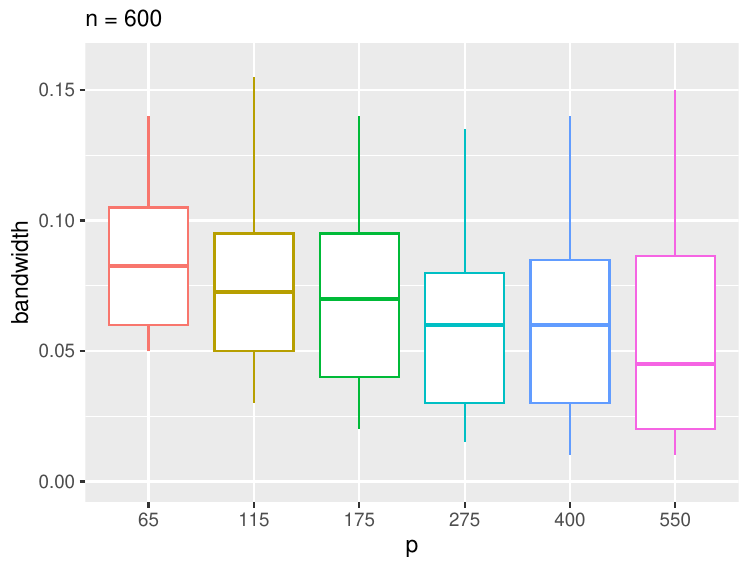}
		\caption{\small Boxplots of $N = 200$ iterations of chosen bandwidths chosen by \textit{leave-one-curve-out cross validation} in the setup of Figure \ref{fig:comp:h:sigma1}.\\[9.5mm]}
		\label{fig:h_boxplot}
	\end{minipage}
\end{figure}

We take independent standard Gaussian errors $\e_{i, j} $ and for the individual deviations we use independent Brownian motions. Simulations are performed in \texttt{R}, and the local polynomial estimator is computed using the \texttt{R}-package \texttt{locpol} with the function \texttt{locPolWeights}. 
Figure \ref{fig:example} contains plots of $\mu_0$, $n=10$ observations at $p=21$ design points as well as the local polynomial estimator of degree 2 based on all as well as only on $p=11$ observations.

In the following, since we have synchronous design points we directly sample the averaged errors and processes, resulting in  $\bar \e_{j,n} \sim \mc N(0, 1/n)$ and for the individual deviation $\bar Z_{j,n} \sim n^{-1/2} B(t),\, t\in [0,1]$ for a Brownian motion $(B(t))_{t\in[0,1]}$.

To determine the sensitivity on the bandwidth, in Figure \ref{fig:comp:h:sigma1}, for sample size $n=600$ and number of design points $p  \in \{65,115, 175, 275,  400,  550)$, based on $N=1000$ iterations we simulate the supremum norm errors on a grid of bandwidths of step size $0.005$, starting at $3/p$. Again we use the local polynomial estimator of degree 2 and complement the results with an interpolation estimator. For large but also very small bandwidths - depending on $p$ - the estimation error in the supremum norm starts again to increase. 


Figure \ref{fig:h_boxplot} contains  boxplots of bandwidths selected by \textit{leave-one-curve-out cross validation}. Comparing the values to Figure \ref{fig:comp:h:sigma1} 
we see that overall reasonable values ares selected, and in the majority of cases the resulting bandwidth seems to result in minimal overall supremum norm error.

\vspace{1mm}

%

\begin{figure}[b!]
	\begin{minipage}[b]{.45\linewidth}
			\centering
			\includegraphics[width=\linewidth]{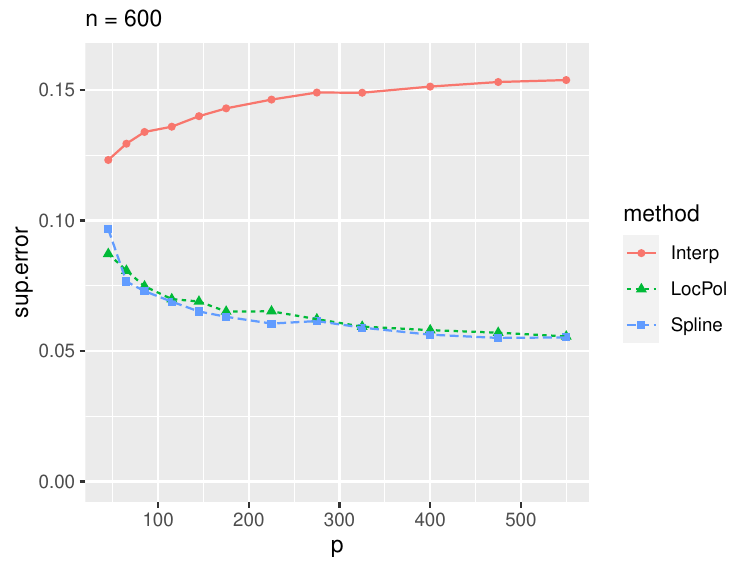}
			\caption{\small Effect of increasing $p$ on overall sup-norm error for local polynomial estimator (green dotted line), smoothing spline (blue dashed line) and interpolation estimator (solid red line).}
			\label{fig:comparison_sup}
	\end{minipage}
	\hspace{0.5cm}
	\begin{minipage}[b]{.45\linewidth}
		\centering
		\includegraphics[width=\linewidth]{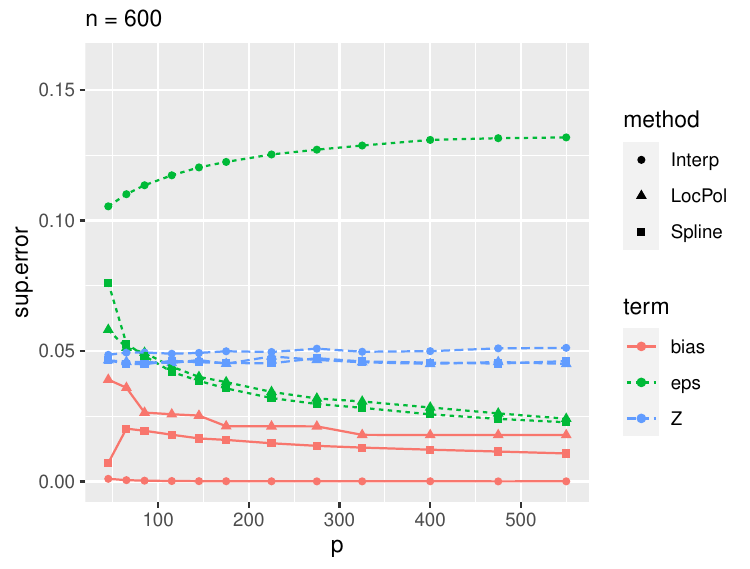}
		\caption{\small Effect of increasing $p$ on contribution of bias (red curves), process terms $Z$ (blue curves) as well as errors (green curves) on the the sup-norm error of the three methods.\\\phantom{h}}
		\label{fig:comparison_decomp}
	\end{minipage}
\end{figure}

Figure \ref{fig:comparison_sup} contains the sup-norm error of the local polynomial estimator with optimally tuned bandwidth, of a 
penalized splines estimator computed using the function \texttt{smooth\_spline} from package \texttt{stats} and optimally tuned smoothing parameter, as well as an interpolation estimator for $n=600$ and a range of values of $p$. We again use $N=1000$ repetitions to determine the expected sup-norm error. While for local polynomial and smoothing splines the sup-norm error first decreases and then stabilizes for increasing $p$, for the interpolation estimator it actually increases with increasing number of design points $p$. Taking a closer look at the contributions of the terms in the error decomposition in \eqref{eq:decomposition} in  
Figure \ref{fig:comparison_decomp}, we observe that while the contributions from the processes $B_i$ are comparable for all three methods, and the bias of the interpolation estimators is quite small, the contribution from the observation errors increases for interpolation as also observed in Remark \ref{rem:errorinterpol}.



\section{Real data illustrations} \label{sec:applications}

In this section we investigate the effect of a coarser discretization in two data examples. Smaller values of $p$ result in fewer data points, which may be advantageous for storage and also for energy efficiency in the measurement process, as long as estimation quality is not too adversely affected. 

To quantify the effect on the sup-norm error, we proceed as follows. We use the full data set to compute the estimator as well as the fast and fair confidence bands proposed in \cite{liebl2019fast}. Here we use the package \texttt{ffscb} from \cite{liebl2019fast} with two intervals in the time spans in our examples. 
Then we compute the estimator for coarser discretizations, that is smaller values of $p$ using only subsets of the full data, and investigate whether the resulting estimator is still fully contained in the confidence bands.

\subsection{Temperature series}

First we consider daily temperatures (per month) for the years 2000 until 2022 in Nuremberg, Germany. The data were  obtained from the \textit{Deutschen Wetter Dienst (DWD)} at \href{https://opendata.dwd.de/climate_environment/CDC/observations_germany/climate/10_minutes/air_temperature/historical/}{[Link]}.
 To remove serial dependence, only the $1, 4, 8, 12, 15, 18, 22, 25$ and  $29^{\text{th}}$ days of each month were used. Data are recorded every 10 minutes, which results in $p = 144$. Some NA's had to be removed which results in varying $n$ for the different months, see Table \ref{tab:monthnh}.  

Figure \ref{fig:weather_curves} contains plots of the observed daily temperature curves of the 
$1, 8, 15, 22$ and $29^{\text{th}}$ day of the month, together with a local polynomial fit. The bandwidth is chosen by leave-one-curve-out cross validation, the results are displayed in Table \ref{tab:monthnh}. 

\begin{table}[h!]
	\begin{tabular}{c|c|c|c|c|c|c|c|c|c|c|c|c}
Month & Jan & Feb & Mar & Apr & May & Jun & Jul & Aug & Sep & Oct & Nov & Dec\\ \hline
$n$ & 178 & 165 & 178 & 178 & 175 & 187 & 184 & 186 & 184 & 184 & 183 & 188 \\ \hline
$h$ & 72.0 & 72.0 & 58.4 & 64.8 & 43.2 & 64.8 & 57.6 & 57.6 & 57.6 & 57.6 & 136.8 & 43.2
\end{tabular}
\caption{Number of observations $n$ and bandwidth $h$ for each month. }\label{tab:monthnh}
\end{table}

\begin{figure}
	\centering
	\includegraphics[width=\linewidth]{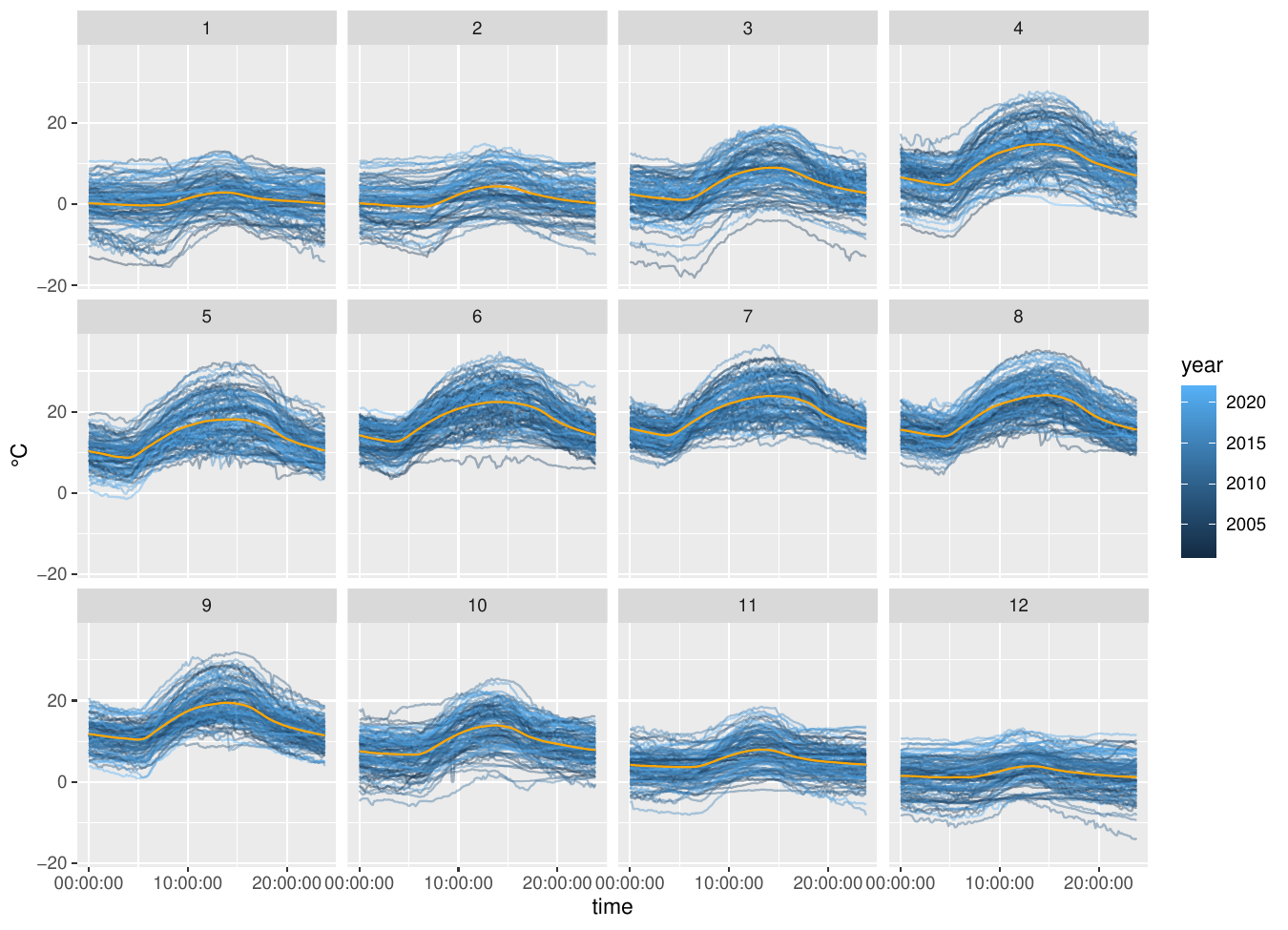}
	\caption{\small Daily temperature curves in Nuremberg, Germany, per month from 2000 until 2022, together with a local polynomial estimate of the mean curve (yellow curve).}
	\label{fig:weather_curves}
\end{figure}


%

\begin{figure}
	\centering 
	\includegraphics[width=10cm]{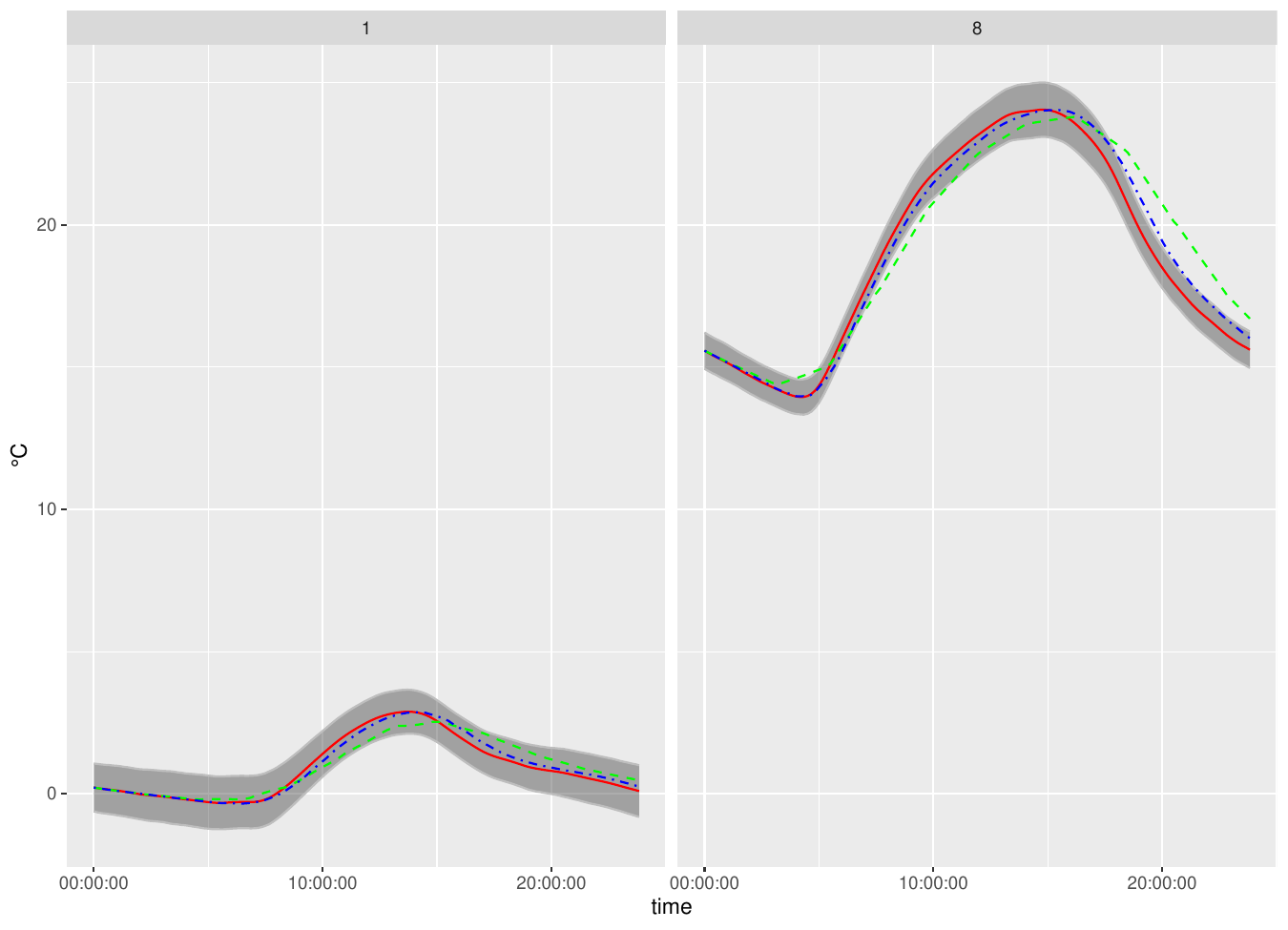}
	\caption{\small LP estimator with $p=144$ (10-minute data, red line), $p=24$ (one hour data, blue dotted/dashed line) and $p = 12$ (two hour data, green dashed line). Confidence bands for the $p=144$ curve from \cite{liebl2019fast} for January and August.}
	\label{fig:weather_conf_bands_jan_aug}
\end{figure}

Next, we use the full data to compute the fast and fair confidence bands proposed in \cite{liebl2019fast} with two time intervals. Finally, we compute the estimator for a coarser discretization for one- and two hour intervals, and investigate whether these curves are still contained in the confidence bands. The results for January and August are displayed in Figure \ref{fig:weather_conf_bands_jan_aug}. While in January, both curves are fully contained in the confidence band, this is not the case in August due to higher temperature variations in day- and night times.  The corresponding plots for all months are given in Figure \ref{fig:weather_conf_bands} in the appendix.  

%


\subsection{Biomechanics}

Second we revisit a data set of a sports biomechanics experiment from \citet{liebl2019fast}. It contains $n=18$ pairs of torque curves in Newton metres (N/m)
standardized by the bodyweight (kg) for runners who wear first, extra cushioned running shoes and second, certain normal cushioned
running shoes. The time scale is normalized to $t \in [0, 100]$ with $p=200$ grid points. See \cite{liebl2019fast} for further details. Figure \ref{fig:diff_curves_bio} contains the curves of pairwise difference for each runner together with a local polynomial estimate of the corresponding mean curve. 
The bandwidth is selected by leave-one-curve-out cross validation, resulting in $h = 5.5$.  

Figure \ref{fig:disc_bio} contains the confidence band computed with the methods from \cite{liebl2019fast} using two intervals in the time span, together with mean curve estimates based on $p=50$ as well as $p=20$ data points in each row. While the fit with $p=50$ is hardly distinguishable from that based on the full data ($p=200$), the estimate based on only $p=20$ is outside the confidence band in the critical starting phase at about $t=6$.  
%

\begin{figure}
	\begin{minipage}[b]{.45\linewidth}
		\centering
		\includegraphics[width=\linewidth]{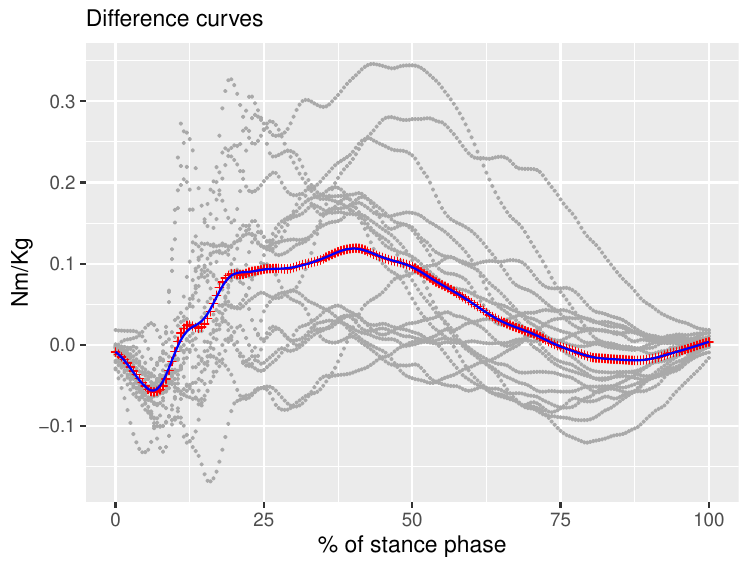}
		\caption{\small Curves of pairwise differences of torque curves with different shoes for each runner, together with an local polynomial estimate of the mean function (solid blue curve)}
		\label{fig:diff_curves_bio}
	\end{minipage}
\hspace{0.5cm}
	\begin{minipage}[b]{.45\linewidth}
		\centering
		\includegraphics[width=\linewidth]{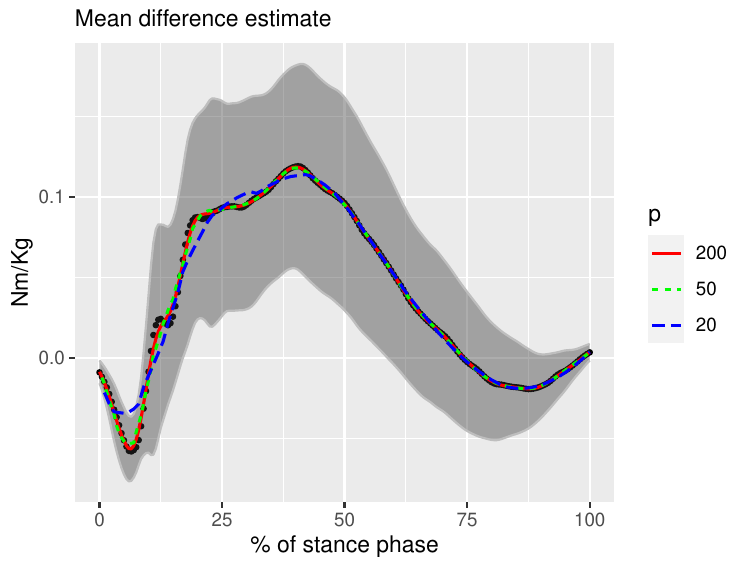}
		\caption{\small Estimated mean difference curves using $p=200$ (solid red), $p=50$ (dotted green) and $p=20$ (dashed blue) design points, and a uniform confidence band computed from the full data}
		\label{fig:disc_bio}
	\end{minipage}	
\end{figure}

%
%

\section{Discussion} \label{sec:discuss}

Let us indicate how to extend our results to more general design assumptions. 
First, we note that if the design satisfies Assumption \ref{ass:design1} and the weights satisfy Assumption \ref{ass:weights}, the upper bound \eqref{eq:upperboundone} in Theorem \ref{theorem:estimation:rates} follows. Here $\bs{ p^1}$ is interpreted as the total number of design points. While the range of admissible bandwidths has to be adapted according to the design under investigation, no Cartesian product structure is required.

Further, the upper bounds in Theorem \ref{theorem:estimation:rates} and even the CLT in Theorem \ref{cor:asymptotic:normality} can be extended to asynchronous design (e.g.~design depending on the observation row $i$) if the number of points in each row is of the same order. Here we use the linear estimator \eqref{eq:linest:mu:bar} in each row, 
\begin{equation}\label{eq:linest:mu:bari}
	\mulei(\bs x) = \sum\limits_{\bs j= \bs 1}^{\bs p_i}   w_{\bs j,\bs p_i,i}(\bs x;h; \bs x_{\bs 1,i},\dotsc,\bs x_{\bs p_i,i}) \, Y_{\bs j} \,, 
\end{equation}
and average to obtain the estimator of $\mu$, potentially using a weighted average \citep{zhang2016sparse}. If the weights $w_{\bs j,\bs p_i,i}$ in each row satisfy Assumption \ref{ass:weights}, with constants $\Cmax, \Clip > 0$ and Assumption \ref{ass:design1} with $\Ccard>0$, which may be chosen independently of $i$, we may obtain results similar to Theorem \ref{theorem:estimation:rates} and Theorem \ref{cor:asymptotic:normality}, where the $n^{-1/2} $ - rate in the dense setting in Theorem \ref{theorem:estimation:rates} follows from the more elaborate manageability argument used in the proof of Theorem \ref{cor:asymptotic:normality}. 

The \textit{smooth first then average} approach discussed above for the asynchronous setting cannot work in the truly sparse case, where the coordinates of $\bs p$ remain bounded. 
Here an alternative is pooling, that is, forming a single linear estimator from all observations potentially with weighting. This has been worked out in detail in \citep{li2010uniform, zhang2016sparse}. In the dense case, however,  only a $ \log n / \sqrt n$ - rate is obtained. Thus, the parametric rate together with the  central limit theorem in supremum norm for  dense designs seem not to be  developed for pooling estimators yet. 



\section{Proofs} \label{sec:main:proofs}

\subsection{Proof of Theorem \ref{theorem:estimation:rates}} \label{sec:proof:theorem:rates}

From the error decomposition \eqref{eq:decomposition} we obtain the bound
\begin{align}
	& \, \sup_{\mu\in\hclass, \, Z \in \Pclass} \E_\mu \Big[ \normib{ \mule - \mu} \Big]\nonumber\\
	\leq & \, \sup_{\mu\in\hclass} \normib{I_1^{\bs p,h}}
	+ \E  \Big[\normib{I_2^{n, \bs p,h}}  \Big]+\sup_{Z \in \Pclass} \E  \Big[ \normib{I_3^{n, \bs p,h}}  \Big] \,.\label{eq:errordecomp}
\end{align}
Theorem \ref{theorem:estimation:rates} follows from \eqref{eq:errordecomp} together with Lemma \ref{theorem:estimation:rates1} below. Note that 
property \ref{ass:weights:sup} in Assumption \ref{ass:weights} together with Assumption \ref{ass:design1} imply the following. 
\begin{enumerate}[label=\normalfont{(W\arabic*)},leftmargin=9.9mm]
		\setcounter{enumi}{4}
			\item 	We have  $ \displaystyle \sum_{\bs 1\leq \bs j \leq \bs p} \big|\wjb xh \big| \leq \Csum$, $\bs x \in T$. \label{ass:weights:sum}
\end{enumerate}
\begin{lemma} \label{theorem:estimation:rates1} 
	In model \eqref{eq:model} under Assumption \ref{ass:model}, consider the linear estimator $\mule$ in \eqref{eq:linest:mu:bar} with error decomposition in \eqref{eq:errordecomp}.  
	\begin{enumerate}
		\item If the weights satisfy \ref{ass:weights:polynom} with $\gamma=\floor \alpha$, \ref{ass:weights:vanish} of Assumption \ref{ass:weights} as well as \ref{ass:weights:sum}, then 
		\begin{align*}
			\sup_{\mu\in\hclass } \normib{I_1^{\bs p,h}} = \Oop{h^\alpha} \, .
		\end{align*}
		
		\item If the weights satisfy \ref{ass:weights:sup}, \ref{ass:weights:lipschitz}  of Assumption \ref{ass:weights} and Assumption \ref{ass:design1}, then 
		\begin{align*}
			\E\Big[\normib{I_2^{n, \bs p,h}} \Big] = \Oop{\sqrt{\frac{\log (h^{-1})}{n\bs p^{\bs 1}h^d}}} \,.
		\end{align*}
		
		\item If the weights satisfy \ref{ass:weights:sum} then 
		\begin{align*}
			\limsup_{n,\bs p \to \infty}\, \sup_{h \in (\bs c/{\bs p}, h_0]} \sup_{Z \in \Pclass}\, \E\Big[ \big\lVert \sqrt n \, I_3^{n, \bs p,h} \big\rVert_\infty^2 \Big] < \infty \,.
		\end{align*}	
		In particular,
		\begin{align*}
			\normib{  I_3^{n, \bs p,h} }= \Oph{n^{-\frac{1}{2}} } \,.
		\end{align*}
	\end{enumerate}
\end{lemma}

\begin{proof}[Proof of Lemma \ref{theorem:estimation:rates1}]
	\textit{i):} This is a well-known result in nonparametric curve estimation by using Taylor expansion and the properties of the weights. The proof can be found in the Appendix, Section \ref{sec:proof:bias}.

	\medskip

\textit{iii):} 
For processes $Z$ in the class $\Pclass$, we have that 
$\frac1{\sqrt n} \, \abs{ \sum_{i=1}^n\,Z_i(\bs{x}) } $ is uniformly bounded in expected value. For example, by using the inequality \eqref{eq:maxboundpollard} in the appendix, which is \citet[Section 7, display (7.10)]{pollard1990empirical}, we explicitly have that 
\begin{equation}\label{eq:boundsup}
	\expec\Big[\sup_{\bs x\in [0,1]^d} \, \frac1{\sqrt n} \, \abs{ \sum_{i=1}^n\,Z_i(\bs{x}) }\Big] \leq \text{const}\, (C_Z/\beta_0)^{1/2},
\end{equation}
see Lemma \ref{lem:amxboundhoelder} below. 
Therefore, by using property \ref{ass:weights:sum} of the weights we obtain 
\begin{align*}
	\E\Big[ \big\lVert \sqrt n \, I_3^{n,\bs p, \bs h} \big\rVert_\infty \Big] & = \expec\bigg[\sup_{\bs x\in [0,1]^d} \sqrt n\, \Abs{\sum_{\bs j = \bs1}^{\bs  p}\wjb xh \bar Z_n(\bs{x_j}) }\bigg]\\
	& \leq \Csum \,\expec\bigg[\sup_{\bs x\in [0,1]^d} \abs{ \frac1{\sqrt n}\, \sum_{i=1}^nZ_i(\bs{x}) }\bigg]\tag{by \ref{ass:weights:sum}}\\
	& \leq \text{const}\,  (C_Z/\beta_0)^{1/2} < \infty. \tag{by \eqref{eq:boundsup}}
\end{align*}

	\medskip
	
	\textit{ii):} 
 	We shall apply Dudley's entropy bound \citep[Corollary 2.2.8]{van1996weak}. To this end, note that by Assumption \ref{ass:model}, $\bar\e_{\bs 1},\ldots,\bar \e_{\bs p}$ are independent, centered, sub-Gaussian random variables with parameters bounded by $\zeta^2\sigma^2/n>0$ and $\E[\bar \e_{\bs j, n}^2]=\sigma_{\bs j}^2/n$, $\bs j=\bs 1,\dotsc,\bs p$. 
 	%
 	Therefore, the process
 	$$S_{n,\bs p, h}(\bs x) = \sqrt{n\bs{p^1}h^d}\,{I_2^{n,\bs p, h}}(\bs x)$$ 
	is a sub-Gaussian process w.r.t.~the semimetric 
	 \begin{align}
		\dd_{S}^2(\bs x, \bs y) & : =  \zeta^2\, \sigma^2\, \bs{p^1}h^d \sum_{\bs j = \bs 1}^{\bs p}\big( \wjb{x}{h} -\wjb{y}{h} \big)^2  \nonumber \\
		&\leq M^2 \bigg(\frac{\norm{\bs x - \bs y}_\infty}{h}\wedge 1\bigg)^2
		\leq M^2\,, \label{eq:metric:Stilde}
	\end{align}
	where $M = \sqrt{2\Ccard} \Clip \, \zeta\, \sigma$, and we used the Lipschitz continuity of the weights, \ref{ass:weights:lipschitz}.  Now, given $0 < \delta < M$ to bound the covering number $	N\big( [0,1]^d,\delta;\dd_S \big)$ and hence the packing number $	D\big( [0,1]^d,\delta;\dd_S \big)$ of $[0,1]^d$ w.r.t.~$\dd_S$, note that from \eqref{eq:metric:Stilde}, 
	 \begin{align*}
		\dd_S(\bs x, \bs \tau) \leq \delta \qquad \Longrightarrow  \qquad \normi{\bs \tau_{j}-\bs x}\leq \frac{\delta h}M, \qquad \bs x, \bs \tau \in [0,1]^d.
	\end{align*}
	We obtain
	\begin{equation}
		D\big( [0,1]^d,\delta;\dd_S \big) \leq N\big( [0,1]^d,\delta/2;\dd_S \big) \leq  \bigg(\frac {2M}{\delta {h}} \bigg)^d . \label{eq:bound:packing}
	\end{equation}
	By observing from \eqref{eq:metric:Stilde} that the diameter of $[0,1]^d$ under $\dd_S$ is upper bounded by $2 M$, Dudley's entropy bound implies that 
  	\begin{align} \label{eq:integral:packing}
 	\E\Big[\sup_{\bs x \in [0,1]^d} \big|S_{n,\bs p, h}(\bs x)\big| \Big] \leq \E\Big[ \big| S_{n,\bs p, h}(\bs x_0) \big| \Big]+ \Csum \int_0^{2 M} \sqrt{\log \Big( D\big( [0,1]^d,\delta;\dd_S \big) \Big)} \,\dx \delta
 \end{align}
 for fixed  $\bs x_0 \in [0,1]^d$ and a universal constant $\Csum>0$. 	Using \eqref{eq:bound:packing}, the integral in the second term is bounded by 
 \begin{align*}
 	 & \int_0^{2M} \sqrt{\log\big((2M / (\delta h)^d) \big)} \dx \delta
 	 = \frac{2 M\,d}{h} \int_0^{h} \sqrt{\log(\delta^{-1})} \,\dx \delta \\
 	\leq \, & 2\,  M\,d \bigg(\sqrt{-\log (h)}-\frac{1}{2\sqrt{-\log (h)}}\bigg) 
 	=\Oop{ \sqrt{-\log (h)} } \,.
 \end{align*} 
 For the first term in \eqref{eq:integral:packing}, apply Jensen's inequality to obtain
 \begin{align*}
 	\E\Big[ \big| S_{\bs p}(\bs x_0) \big| \Big]^2 &\leq \E\big[\widetilde{S}_{\bs p}^2(\bs x_0)\big] \leq \sigma^2 \bs{p^1}h^d \sum_{\bs j = \bs 1}^{\bs p} \wjb{x_0}{h}^2 \leq \sigma^2 \Csum \Cmax < \infty
 \end{align*}
 by 
 \ref{ass:weights:sup} and \ref{ass:weights:sum}. This concludes the proof of the theorem.  	 	
\end{proof}

\subsection{Proof of Theorem \ref{theorem:optimality} (optimality)}\label{sec:proof:theorem:optimality}

\begin{proof}[Proof of Theorem \ref{theorem:optimality}.]

In the proof we rely on the reduction to hypothesis testing as presented e.g.~in \citet[section 2]{tsybakov2008introduction}.

For the lower bound $\bs p_{\min}^{-\alpha}$, using the method of two sequences of hypotheses functions we set $\mu_0=0$ and construct $\mu_1 = \mu_{1, \bs p}$ such that $\normi{\mu_0 - \mu_{1, \bs p}} \geq c\, \bs p_{\min}^{-\alpha}$ for some constant $c>0$ and that $\mu_{1, \bs p}(\bs x_{\bs j})=0$ at all design points $\bs x_{\bs j}$, so that the distribution of the observations for $\mu_0$ and $\mu_1$ coincide.  

If $r$ is such that $p_r = \bs{p}_{\min}$, observing Assumption \ref{ass:designdensity} for a constant $\tilde L>0$ to be specified we set
 $$g_{\bs p}(\bs x) = g(\bs x)= \tilde L\, \bigg(\frac {1}{p_r {f_{\max}}_r}\bigg)^{\alpha} 
		\exp\bigg(-\frac1{1-x_r^2}\bigg)\one_{\{|x_r| < 1\} }, \qquad \bs x = (x_1, \ldots, x_d),$$
	and for some fixed $ 1 \leq l \leq p_r - 1$ let
 $$ \mu_1(\bs x)= g\big((2p_r {f_{\max}}_r)\, (x_r -(x_{r,l}+ x_{r,l+1})\,/\,2)\big). $$
For the design points $x_{r,l}$ in the $r^{\mathrm{th}}$ coordinate, by Assumption \ref{ass:designdensity} the distance is at least $p_r {f_{\max}}_r$, and therefore by the definition of $g$ and the Cartesian product structure of the design it follows that $\mu_{1, \bs p}(\bs x_{\bs j})=0$ at all design points. Further, if the $r^{\mathrm{th}}$ coordinate is $(x_{r,l}+ x_{r,l+1})\,/\,2$, $\mu_{1, \bs p}$ takes the value $p_r^{- \alpha}\, \tilde L\, e^{-1} / {f_{\max}}_r^\alpha$, so that $\normi{\mu_0 - \mu_{1, \bs p}} \geq c\, p_{\min}^{-\alpha}$ holds true. Finally, using the chain rule and the fact that all derivatives of the bump function in the definition of $g$ are uniformly bounded, $\mu_1$ is $\alpha$ - Hölder smooth (indeed all partial derivatives except in direction $r$ vanish) with constant proportional to $\tilde L$, which can be adjusted to yield the Hölder norm $L$. This concludes the proof for the lower bound $\bs p_{\min}^{-\alpha}$.

    Let us turn to the lower bound of order $(\log(n\,\bs{p^1})/(n{\bs p ^1}))^{\alpha/(2\alpha+d)}$, which corresponds to the minimax rate in $d$ dimensions in the conventional nonparametric regression model with $n{\bs {p ^1}}$ design points. Taking $Z = 0\;(\in \Pclass)$  we can adapt the argument in \citet[Theorem 2.10]{tsybakov2008introduction}. Indeed, by sufficiency and taking averages using the normal distribution we could reduce the analysis to a single row with errors distributed as $\mathcal N(0, \sigma_0^2/n)$.   
    For convenience we provide the details, and do not use sufficiency, so that the argument can be extended to further error distributions  \citep[Assumption B, Section 2.3]{tsybakov2008introduction}.   
    %
    %
      For sufficiently small $c_i$, $i=0,1$ to be specified set
    \[N_{n,\bs p}= \ceil{c_0\bigg(\frac{n\bs{p^1}}{\log(n\bs{p^1})}\bigg)^{\frac1{2\alpha+d}}}, \qquad h_{n,\bs p} = N_{n,\bs p}^{-1}, \qquad s_{n,p}\defeq  c_1\, h_{n,\bs p}^{\alpha}.\]
    For $\tilde L$ to be specified let
    $$\tilde g(\bs x) = \tilde L h_{n,\bs p}^{\alpha} \exp\Big(-\big(1-\norm{\bs x}_2^2\big)^{-1}\Big)\one_{\{\normz{\bs x}< 1\} }$$
    and for $$\bs l = (l_1, \ldots, l_d) \in \{1, \ldots, N_{n,\bs p} \}^d$$ set
    $$ \mu_{\bs l}(\bs x) = \tilde g\big(2\, (\bs x - \bs z_{\bs l})/h_{n,\bs p}\big), \qquad \bs z_{\bs l}= \big(l_1-1/2,\ldots,l_d-1/2\big)/N_{n,\bs p}. $$

      Using the chain rule one checks that $\mu_{\bs l} \in \hclass$ for suitable choice of $\tilde L$ depending on $L$ and $c_0$. Further, by construction the $\mu_{\bs l}$ have disjoint supports in $[0,1]^d$, so that $\normi{\mu_{\bs l} - \mu_{ \bs r}}\geq 2\,s_{n,\bs p}$ for all $\bs l \neq \bs r$, and $c_1$ sufficiently small depending on $c_0$. 
 %
	
	To apply \citet[Proposition 2.3]{tsybakov2008introduction} and thus to obtain $s_{n,\bs p}$ as lower bound for the rate of convergence, it remains to show for a $\kappa \in (0,1/8)$ and $n$ and $\bs p$ large enough that
%
 %
	\begin{equation}\label{eq:boundkulback}	
\frac{1}{N_{n,\bs p}^d} \sum_{\bs l}   \KL( \Prob_{\mu_{\bs l}}, \Prob_0) \leq d\, \kappa \, \log(N_{n,\bs p})\qquad \text{ for each } \bs l,
  \end{equation}
    where $\Prob_{\mu_{\bs l}}$ is the (joint normal) distribution of the observations for the mean function $\mu_{\bs l}$. 
	To this end, since Kullback-Leibler divergence for product measures is additive, we can estimate
 %
	%
	\begin{align*}
		\frac{1}{N_{n,\bs p}^d} \sum_{\bs l}  \KL(\Prob_{\mu_{\bs l}},\Prob_0) 
		&= \frac{1}{N_{n,\bs p}^d}\, \frac{n}{2 \sigma_0^2}\,\sum_{\bs l}\, \sum_{\bs 1\leq \bs j \leq \bs p}\,  \mu_{\bs l}^2(\bs{x_j}) \tag{normal distr.}\\
		&\leq \frac{n}{2 \sigma_0^2}\, \frac{\tilde L^2}{e^2}\, h_{n,\bs p}^{2\alpha+d}\sum_{\bs 1\leq \bs j \leq \bs p}\, \sum_{\bs l}\,\one_{\{2\normz{\bs{x_j}-\bs z_l} < h_{n,\bs p}\}} \tag{constr.~of $\mu_{\bs l}$}\\
		& \leq \frac{n}{2 \sigma_0^2}\, \frac{\tilde L^2}{e^2}\,\frac{h_{n,\bs p}^{2\alpha+d}}{2^{2\alpha-d}}
		\,\bs{p^1}\\
        & \leq \text{const.}\,  \log(n\, \bs p^{\bs 1}), \tag{choice of $h_{n,\bs p} = N_{n,\bs p}^{-1}$}
	\end{align*}
 which implies \eqref{eq:boundkulback} and completes the proof for the lower bound $(\log(n\,\bs{p^1})/(n{\bs p ^1}))^{\alpha/(2\alpha+d)}$.

Finally for the lower bound of order $n^{-1/2}$ we may follow the argument in \citet{cai2011optimal}: Take $\mu$ to be a constant function, and $Z_i$ to be a random constant, which is centered normally distributed with variance $\leq C_Z$. Since the Kullback - divergence between two normal distributions with different mean and equal variance only decreases under convolution with a further centered normal distribution, it suffices to obtain the lower bound $n^{-1/2}$ in the model without errors $\e_{i, \bs j}$. But then it follows as the known rate of convergence for estimating a normal mean in a sample of size $n$. This concludes the proof of Theorem \ref{theorem:optimality}.  
\end{proof}







\subsection{Manageability and a maximal inequality from \citet{pollard1990empirical}} \label{sec:maxinequa}

We recall the concept of manageability of a triangular array of processes from \citet[Section 7]{pollard1990empirical}, which we require to prove \eqref{eq:boundsup} as well as Theorem \ref{cor:asymptotic:normality}.


For a bounded subset $B \subset \R^n, n\in\N$, we call a vector $ \bs\Phi = (\Phi_1,\ldots, \Phi_n)^\top\in \R^n$ an \textit{envelope} for $B$ if $\abs{b_i} \leq \Phi_i$, for all $\bs b = (b_1,\ldots, b_n)^\top \in B$.
For a triangular array $(X_{n,i}(x))_{n \in \N, 1\leq i \leq k_n}\colon \Omega\times [0,1]^d \to \R$ of stochastic processes with independent rows we call the sequence $\bs \Phi_n$ of $k_n$-dimensional random vectors a sequence of \textit{envelopes} of $(X_{n,i})$ if $\bs \Phi_n(\omega)$ is an envelope of the set $F_n(\omega)\defeq \big\{(X_{n,1}(\omega, \bs x) ,\ldots,X_{n,k_n}(\omega, \bs x))^\top \mid \bs x \in [0,1]^d\big\}$ for $n  \in \N$ and $\omega \in \Omega$.

Further  we call a triangular array of processes $(X_{n,i}(\bs x)), i = 1, \ldots, k_n, \bs x \in [0,1]^d$ \textit{manageable} (with respect to its envelope $\Phi_n$) \citep[see Definition 7.9, p. 38]{pollard1990empirical} if there exists a deterministic function $\lambda\colon[0,1] \to  \R$, the \textit{capacity bound}, for which
\begin{enumerate}
	\item $\int_0^1 \sqrt{\log ( \lambda(x))}\dx \varepsilon < \infty,$
	\item $ D(\varepsilon\norm{\bs \alpha \circ \Phi_n(\omega)}_2, \bs \alpha \circ F_n(\omega)) \leq \lambda(\varepsilon)$ pointwise for $\varepsilon \in [0,1]$ and all $\bs \alpha \in \R_+^n$, $n \in \N$ and all $\omega\in \Omega$
\end{enumerate}
where $ D(\varepsilon\norm{\bs \alpha \circ \Phi_n(\omega)}_2, \bs \alpha \circ F_n(\omega))$ is the packing number and $\bs x \circ \bs y$ denotes the Hadamard (component-wise) product of two vectors of equal dimension. 
A sequence $(X_{i})$ of processes \textit{manageable} if the array defined by $X_{n,i} = X_i $ for $i\leq n$ is manageable. 
To check manageability we shall use the following Lemma.

\begin{lemma}\label{lem:manageability}
	Consider the triangular array $(X_{n,i})_{n\in \N, 1\leq i \leq k_n}$ of stochastic processes on $[0,1]^d$ with a suiting sequence of envelopes $(\bs \Phi_n)_{n \in \N}$. If there are constants $K_1, K_2, \kappa \in \R^+$ such that for all $\bs x,\bs x^\prime \in [0,1]^d, n\in\N $ and $\epsilon>0 $ it holds that
	\begin{equation}\label{eq:measucondprocess}
		\norm{\bs x - \bs x^\prime}_2 \leq K_1 \epsilon^\kappa \Rightarrow \big|X_{n,i}(\bs x) - X_{n,i}(\bs x^\prime)\big| \leq \epsilon K_2 \Phi_{n,i}, 
	\end{equation}
	where $i =1,\ldots, k_n$, and $\bs \Phi_n = (\Phi_{n,1},\ldots, \Phi_{n,k_n})^\top$, then we have for the covering number that
	\begin{equation*}
		N(\epsilon \norm{\bs \alpha \circ \bs \Phi_n}_2, \bs \alpha  \circ F_n) \leq \bigg(\frac{K_2^\kappa}{2K_1}\epsilon^{-\kappa} + 2\bigg)^d =: \lambda_\kappa(\varepsilon).
	\end{equation*}
	Further it holds that 
	\begin{equation}\label{eq:theconstant}
		\Lambda_\kappa \defeq \int_0^1 \sqrt{\log(\lambda_\kappa(\varepsilon))}\, \dx \epsilon < \infty
%
	\end{equation}
\end{lemma}
\begin{proof}[Proof of Lemma \ref{lem:manageability}]
	Let $\tau_1,\ldots, \tau_N \in [0,1]^d$ be such that for $\bs x\in [0,1]^d$ there exists $\tau_m$ such that $\norm{\bs x - \tau_m}_2^2 \leq \frac{K_1}{K_2^\kappa}\,\epsilon^\kappa$. For such $\bs x$ and $\tau_m$, from \eqref{eq:measucondprocess} we obtain
	\begin{align*}
		\abs{X_{n,i}(x) - X_{n,i}(\tau_m)} \leq \epsilon \Phi_{n} \quad \text{and hence} \quad \norm{\alpha \circ X_{n}(x) - \alpha \circ X_{n,i}(\tau_m)}_2 \leq \epsilon \norm{\alpha \circ \Phi_{n}}_2
	\end{align*}
for all $\alpha \in \R^d$. 
	Therefore the amount of balls with radius $\epsilon\, \norm{\alpha \circ \Phi_{n}}_2$ at points $T_m = \alpha \circ X_{n}(\tau_m)$ needed to cover $\alpha \circ F_n$ is upper bounded by the number of balls with radius $K_1\epsilon^\kappa K_2^{-\kappa}$ needed to cover $[0,1]^d$, which in turn is upper bounded by 
	\[ \lambda_\kappa^d(\epsilon) \defeq \bigg( \frac{K_2^\kappa}{2K_1\epsilon^\kappa} + 2\bigg)^d\,.\]
	The integral in \eqref{eq:theconstant} is finite since
	\begin{align*}
		\int_0^1 \sqrt{\log(\lambda_\kappa^d(\epsilon))} \dx \epsilon & \leq \sqrt d \int_0^1 \bigg(\sqrt{\log \Big( \frac{K_2^\kappa}{2\,K_1}\Big)} + \sqrt{- \kappa\,\log(\epsilon)} + \sqrt{\log\Big(1+ \frac{K_1}{K_2^\kappa}\Big)}\,\bigg)\, \dx \epsilon \\
		& \leq \sqrt{\frac{d\,\kappa \,\pi}2} + \text{const} < \infty\,.
	\end{align*}
\end{proof}

Let us state the maximal inequality \citet[Section 7, display (7.10)]{pollard1990empirical}. 
\begin{lemma}\label{lem:maximalInequalityExpection}
	Let $(X_{n,i})$ be a triangular array of stochastic processes with a suitable sequence of envelopes $(\Phi_n)$. Further let $\lambda\colon [0,1] \to \R$ be the capacity bound such that $(X_{n,i})$ is manageable with respect to $(\Phi_n)$ and $\lambda$. 
	We write 
	\begin{equation*}
		S_n(\bs x) \defeq \sum_{i=1}^{k_n} X_{n,i}(\bs x) \qquad \textit{and} \qquad M_n(\bs x) \defeq \expec\big[S_n(\bs x)\big],
	\end{equation*}
	then for all $1\leq p <\infty$ there exists a constant $\tilde C_p>0 $ such that 
	\begin{equation}\label{eq:maxboundpollard}
		\expec\big[\norm{S_n - M_n}^p_\infty\big] \leq \tilde C_p \, \Big(\int_0^1 \sqrt{\log ( \lambda(\varepsilon))}\dx \varepsilon\Big)^p \expec\big[\norm{\bs \Phi_n}_2^p\big].
	\end{equation}
\end{lemma}
From Lemmas \ref{lem:manageability} and \ref{lem:maximalInequalityExpection} we may conclude the bound \eqref{eq:boundsup}, which we state in the following lemma. 
\begin{lemma}\label{lem:amxboundhoelder}
For $Z, Z_1, \ldots, Z_n$ i.i.d.~stochastic processes on $[0,1]^d$ with $Z \in \Pclass$ we have the bound
\begin{equation}\label{eq:boundsup1}
	\expec\Big[\sup_{\bs x \in [0,1]^d}\, \frac1{\sqrt n} \, \abs{ \sum_{i=1}^n\,Z_i(\bs{x}) }\Big] \leq \big(2\, \tilde C_2\, C_Z\, \Lambda_{1/\beta_0} \big)^{1/2}
\end{equation}
where $\Lambda_{1/\beta_0}$ is given in \eqref{eq:theconstant}, and $\tilde C_2$ in \eqref{eq:maxboundpollard}. 
\end{lemma}
\begin{proof}
	From the Hölder continuity of the paths \eqref{eq:hoeldercontpathsZ}, Lemma \ref{lem:manageability} implies manageability of each process $n^{-1/2} Z_i$  w.r.t.~to the envelope $n^{-1/2} (M_i + \abs{Z_i(\bs 0)}), i = 1,\ldots,n$, where we may take $K_1 = K_2 = 1$ and $\kappa = 1/\beta_0$ uniformly over $\Pclass$. The bound follows from Jensen`s inequality and \eqref{eq:maxboundpollard} with $p=2$. 
\end{proof}

\subsection{Proof of Theorem \ref{cor:asymptotic:normality} (asymptotic normality)} \label{sec:proof:theorem:asympnorm}

\begin{proof}[Proof of Theorem \ref{cor:asymptotic:normality}]
	
	In view of the error decomposition \eqref{eq:decomposition}, Lemma \ref{theorem:estimation:rates} together with the assumptions on $\bs p$ and $n$ and the choice of the smoothing parameter $h$ imply
	\begin{align*}
		\sup_{\mu \in \hclass } \sqrt n \, \normb{I_1^{\bs p,h}}_\infty \to  0 \quad \text{and} \quad \sqrt n \, \normb{I_2^{n,\bs p, h}}_{\infty} \stackrel{\Prob}{\to}  0 \,.
	\end{align*}
	We shall apply the functional central limit Theorem in \citet[Theorem (10.6)]{pollard1990empirical}	to the triangular array of processes $(X_{ni})_{n \in \N, 1 \leq i \leq n}$ given by
\begin{equation*}
	X_{ni}(\bs x)= \frac1{\sqrt n} 	\sum_{\bs j = \bs 1}^{\bs p} \wjb xh \, Z_i(\bs x_{\bs j})\quad  \text{and}\quad  S_{n}(\bs x) \defeq \sum_{i=1}^n X_{ni}(\bs x),  \qquad \bs x \in [0,1]^d \,,
\end{equation*}
to obtain 
	\begin{align}\label{eq:clthelp}
		\sqrt n \, I_3^{n,\bs p,h}  = S_n(\bs x) \stackrel{D}{\longrightarrow} \ \GG(\bs 0,c) \, .
	\end{align}
	With Slutsky's Theorem \citep[Example 1.4.7]{van1996weak} we then obtain the assertion of Theorem \ref{cor:asymptotic:normality}. 

	For \eqref{eq:clthelp} we have to check the conditions i) - v) of \citet[Theorem (10.6)]{pollard1990empirical}.
	
	i). We show manageability of the processes $X_{n} = (X_{n1},\ldots,X_{nn})$ with respect to $\Phi_{n}\defeq(\Phi_{n1},\ldots,\Phi_{nn})$, with
	$$ \Phi_{ni} \defeq \frac{\Csum}{\sqrt n} \big( \abs{Z_i(\bs 0)} + M_i\big)\,,$$
	where the random variables $M_i$ are given \eqref{eq:hoeldercontpathsZ}.
	We apply Lemma \ref{lem:manageability} and distinguish the cases $\epsilon > h^\beta$ and $\epsilon \leq h^\beta$. For $\epsilon > h^\beta$, using \ref{ass:weights:polynom} we get for $\norm{\bs x - \bs y}_2 \leq \epsilon^{1/\beta}$ that
	\begin{align*}
		\sqrt n \absb{X_{ni}(\bs x)  - X_{ni}(\bs y)} & \leq \abss{\sum_{\bs j=\bs1}^{\bs p} \wjb xh\big( Z_i(\bs{x_j}) - Z_i(\bs x)\big)} + \absb{ Z_i(\bs x) - Z_i(\bs y)} \\ & \Quad + \abss{\sum_{\bs j = \bs 1}^{\bs p} \wjb yh \big(Z_i(\bs y) - Z_i(\bs{x_j})}.
	\end{align*}
	The second term is bounded by 
	\begin{align*}
		\absb{Z_i(\bs x) - Z_i(\bs y)} & \leq M_i  \norm{\bs x- \bs y}_\infty^\beta \leq M_i  \norm{\bs x- \bs y}_2^\beta\leq M_i \epsilon\,
	\end{align*}
	and the first and third terms by
	\begin{align*}
		\abss{\sum_{\bs j = \bs 1}^{\bs p} \wjb xh \big( Z_i(\bs{x_j}) - Z_i(\bs x)\big)} & \leq \Csum\absb{Z_i(\bs{x_j}) - Z_i(\bs x)} \ind_{\norm{\bs x- \bs{x_j}}_\infty \leq h}\\
		& \leq \Csum M_ih^\beta \leq \Csum M_i \epsilon\,.
	\end{align*}
	Thus \eqref{eq:measucondprocess} in Lemma \ref{lem:manageability} follows with $K_1 = 1, K_2 = 3$ and $\kappa = 1/\beta$. For the second case $\epsilon \leq h^\beta$ let $\norm{\bs x - \bs y}_2 \leq (\Clip \Ccard )^{-1}\epsilon^{1+1/\beta}$. Then 
	\begin{align*}
		\absb{X_{ni}(\bs x) - X_{ni}(\bs y)} & \leq \Phi_{ni} \sum_{\bs j = \bs 1}^{\bs p} \absb{\wjb xh - \wjb yh} \\
		& \leq \Phi_{ni} \Clip \, \Ccard  \bigg( \frac{\norm{\bs x - \bs y}}h \wedge 1\bigg) \leq \Phi_{ni} \,  \Clip \, \Ccard \,\epsilon. 
	\end{align*}
	In this case \eqref{eq:measucondprocess} in Lemma \ref{lem:manageability} follows for $K_1 = (H\, \Clip \, \Ccard )^{-1}, K_2 = 1$ and $\kappa = 1+ 1/\beta$. Summarizing we may bound the capacity number by 
	\[ N(\epsilon\, \norm{\alpha \circ \Phi_n}, \alpha \circ F_n) \leq \bigg( \frac{\max(3^{1/\beta}, \Clip \Ccard )}{2}  \, \epsilon^{-(1+\nicefrac1\beta)} + 2\bigg)^d  \defeql \lambda_\beta^d(\epsilon) \]
	which is integrable, see Lemma \ref{lem:manageability}. 
	%
	
	For ii) we calculate $\lim_{n\to \infty}\expec[S_n(\bs x)S_n(\bs x^\prime)]$.  It holds
	\begin{align*}
		\expec[S_n(\bs x)S_n(\bs x^\prime)] & = \sum_{\bs j,\bs k =\bs 1}^{\bs p} \wjb xh \wjb{x^\prime}h \Gamma(\bs{x_j}, \bs{x_k}) \stackrel{n \to \infty}{\longrightarrow} \Gamma(\bs x, \bs x^\prime), 
	\end{align*}
	uniformly for all $\bs x, \bs x^\prime \in [0,1]^d$ since $c \in \mc H_{[0,1]^{2d}}(\gamma, \tilde L)$ as stated in Assumption \ref{ass:model:2}.
	
	For iii) we make use of the fact that the processes $Z_i$ have finite second moment
	\[ \sum_{i=1}^n \expec[\phi_{n,i}^2] \leq 2\,\Csum^2\, \expec[Z_1^2(\bs 0) + M_1^2] < \infty\,.\]
	For iv) it holds for all fixed $\epsilon > 0 $ that 
	$$ \sum_{i=1}^n \expec[\Phi_{n,i}\ind_{\{\Phi_{n,i} > \epsilon\}}] \leq 2\,\Csum\,\expec[(Z_1^2(\bs 0) + M_1^2) \ind_{\{\Phi_{n,i} > \epsilon\}}] \stackrel{n \to \infty}\longrightarrow 0$$
	by the monotone convergence Theorem. 
	
	For v) we calculate 
	\begin{align*}
		\rho_n^2(\bs x, \bs x^\prime)  
		& = \sum_{\bs j, \bs k = \bs 1}^{\bs p} \wjb xh \wkb xh \Gamma(\bs{x_j}, \bs{x_k}) -  2\sum_{\bs j, \bs k = \bs 1}^{\bs p} \wjb{x^\prime}h \wkb xh \Gamma(\bs{x_j}, \bs{x_k}) \\
		& \Quad + \sum_{\bs j, \bs k = \bs 1}^{\bs p} \wjb{x^\prime}h \wkb{x^\prime}h \Gamma(\bs{x_j}, \bs{x_k})  \stackrel{n \to \infty}{\longrightarrow} \Gamma(\bs x,\bs x) -  \Gamma(\bs x,\bs x^\prime) + \Gamma(\bs x^\prime, \bs x^\prime)\,
	\end{align*}
	uniformly for all $\bs x, \bs x^\prime \in [0,1]^d$. Further let $(\bs x_n)_{n \in \N}, (\bs y_n)_{n \in \N} \subset [0,1]^d$ be two deterministic sequences that suffice $\rho(\bs x_n, \bs y_n) \to 0, n\to \infty$. Then it holds
	\begin{align*}
		0 \leq \rho_n(\bs x_n, \bs y_n) & \leq \sup_{\bs x, \bs x^\prime \in [0,1]^d} \absb{ \rho_n(\bs x, \bs x^\prime) - \rho(\bs x, \bs x^\prime)}+ \rho(\bs x_n, \bs y_n) \stackrel{n \to \infty}\longrightarrow 0\,.
	\end{align*}

	Thus the functional central limit Theorem is applicable and yields the stated result. 
%
\end{proof}

\appendix

\section{Proof of Lemma \ref{lemma:locpol:weights}} \label{sec:proof:lemma:locpol:weights}


For the proof of Lemma \ref{lemma:locpol:weights} we start with an auxiliary Lemma concerning the approximate localisation of the design points on each axis and their approximate distance.

\begin{lemma} \label{lemma:design:points:auxiliary}
	Let Assumption \ref{ass:designdensity} be satisfied. Then for all $j = 1,\dotsc,p_k$ and $1 \leq k \leq d$ it follows that
	\begin{equation}
		\frac{j-0.5}{\fkmax  \, p_k} \leq x_{k,j} \leq \frac{{j}-0.5}{\fkmin  \, p_k} \quad \text{and} \quad 1-\frac{p_k-{j}+0.5}{\fkmin  \, p_k} \leq x_{k,j} \leq 1-\frac{p_k-{j}+0.5}{\fkmax  \, p_k} \, ,\label{lemma:design:points:auxiliary:1}
	\end{equation}
	and for all $1 \leq j < l \leq p_k$ that
	\begin{equation}
		\frac{l-j}{\fkmax  \, p_k} \leq x_{k,l}-x_{k,j} \leq \frac{l-j}{\fkmin  \, p_k} \,. \label{lemma:design:points:auxiliary:2}
	\end{equation}
\end{lemma}

\begin{proof}[Proof of Lemma \ref{lemma:design:points:auxiliary}]
	The first assertion follows by the equivalence 
	\begin{align*}
		\int_0^{x_{k,j}} \fkmin  \, \dx t \leq \int_0^{x_{k,j}} f_k(t) \, \dx t \leq \int_0^{x_{k,j}} \fkmax \, \dx t \quad &\Leftrightarrow \quad x_{k,j} \, \fkmin  \leq \frac{j-0.5}{p_k} \leq x_{k,j} \, \fkmax  \,.
	\end{align*}
	Similar equivalences and transformations with
	\begin{align*}
	\int_{x_{k,j}}^1 \fkmin  \, \dx t \leq \int_{x_{k,j}}^1 f_k(t) \, \dx t = 1- \int_0^{x_{k,j}} f_k(t) \, \dx t \leq \int_{x_{k,j}}^1 \fkmax  \,\dx t
	\end{align*}
	and 
	\begin{align*}
		\int_{x_{k,j}}^{x_{k,l}} \fkmin  \, \dx t \leq \int_{x_{k,j}}^{x_{k,l}} f_k(t) \, \dx t = \int_{0}^{x_{k,l}} f_k(t) \, \dx t - \int_0^{x_{k,j}} f_k(t) \, \dx t \leq \int_{x_{k,j}}^{x_{k,l}} \fkmax  \, \dx t
	\end{align*}
	show the other two assertations.
\end{proof}

\begin{lemma} \label{lemma:design:points}
	Suppose that Assumption \ref{ass:designdensity} is satisfied. 
 Then we obtain for all subsets defined by intervals $I_1,\ldots, I_d \subseteq [0,1]$ via $I=I_1 \times \ldots \times I_d \subseteq [0,1]^d$ the estimate
	\[ \sum_{\bs 1 \leq \bs j \leq \bs p} \one_{\{\bs x_{\bs j} \in I\}}\leq 2^d\bs \fmax ^{\bs 1} \prod_{k=1}^{d}\max\big( p_k\,\nu(I_k), 1 \big) \,, \]
	where $\nu(I_k)$ is the Lebesgue measure of $I_k$. In particular this means that Assumption \ref{ass:design1} is satisfied.
\end{lemma}

\begin{proof}[Proof of Lemma \ref{lemma:design:points}]
	Without loss of generality we set $I=[a_1,b_1]\times \cdots \times [a_d,b_d]\subseteq[0,1]^d$. We distinguish two cases with respect to the length of the $k$-interval $[a_k,b_k]$. Let $x_{k,1},\ldots, x_{k,p_k} \in [0,1]$ be the grid points on $k^\text{th}$ axis. If $b_k-a_k<1/p_k$, we can bound the number of design points in the interval $[a_k,b_k]$ by $\ceil{\fkmax} =  \min\{z\in \N \mid z>\fkmax\}$. Otherwise, if we assume that $x_{k,j}$ is the smallest point larger than or equal to $a_k$, the lower bound in \eqref{lemma:design:points:auxiliary:2} would yield to the contradiction 
	\begin{equation*}
		x_{k,j+\ceil{\fkmax }} - x_{k,j}\geq \frac{\ceil{\fkmax}}{\fkmax \, p_k} > \frac1{p_k} \,.
	\end{equation*}
	Hence,
	\begin{equation*}
		\sum_{j=1}^{p_k} \one_{\{x_{k,j} \in [a_k,b_k]\}} \leq \ceil{\fkmax} \leq \frac{2 \fkmax}{p_k} = 2 \fkmax \max\big( p_k\,(b_k-a_k),1\big)
	\end{equation*}
	follows. Conversely, if $b_k-a_k \geq 1/p_k$, the number of points in die corresponding interval is bounded by $\ceil{(b_k-a_k)\fkmax p_k }$ since otherwise we would obtain
	\begin{equation*}
		x_{k,j+\ceil{(b_k-a_k)\fkmax p_k}}-x_{k,j} \geq \frac{\ceil{(b_k-a_k)\fkmax p_k} }{\fkmax \, p_k} >b_k-a_k
	\end{equation*}
	just like before. Again this leads to 
	\begin{equation*}
	    \sum_{j=1}^{p_k} \one_{\{x_{k,j} \in [a_k,b_k]\}} \leq \ceil{(b_k-a_k) \fkmax\,p_k} \leq \fkmax(b_k-a_k)+1 \leq 2 \fkmax \max\big( p_k\,(b_k-a_k),1 \big) \, .
	\end{equation*}
	Together with \(
		\sum_{\bs j=\bs 1}^{\bs p} \one_{\{\bs x_{\bs j} \in I\}}=\prod_{k=1}^d \sum_{j=1}^{p_k} \one_{x_{k,j} \in [a_k,b_k]}\)
	the two cases yield the first claim.
 
     Assumption \ref{ass:design1} follows with $C_4=4^d\bs{f_{\max}^1}$ by
     \begin{align*}
         \sum_{\bs j = \bs 1}^{\bs p} \ind_{\bs{x_j} \in [\bs x -(j,\ldots,h), \bs x+ (h,\ldots,h)]} & \leq 2^d \bs{f_{\max}^1} \,\prod_{k=1}^d \max\big(2\,h\,p_k, 1\big) = 4^d\bs{f_{\max}^1}\,h^d\,\bs{p^1}\,.
     \end{align*}
	
\end{proof}


\textbf{Construction of the weights of the linear estimator.}

\vspace{2mm}

In order to derive the form of the weights $\wjb xh$ of the linear estimator $\mulp$ in \eqref{eq:locpolest} we set
\begin{equation}
	B_{\bs p,h}(x) \defeq \frac1{\bs{p^1}h^d}\sum_{\bs j=\bs 1}^{\bs p} U_h(\bs{x_j}-\bs x)  \, U_h^\top(\bs x_{\bs j}-\bs x) \, K_{h}(\bs x_{\bs j}-\bs x)  \quad \in\R^{N_{m,d}\times N_{m,d}} \label{eq:Bp}
\end{equation}
and
\begin{equation*}
	\bs a_{\bs p,h}(\bs x)\defeq \frac1{\bs{p^1}h^d} \sum_{\bs j=\bs 1}^{\bs p} U_h(\bs x_{\bs j}-\bs x) \, K_{h}(\bs x_{\bs j}-\bs x) \, \oY{\bs j} \quad \in \R^{N_{m,d}} 
\end{equation*}
for $\bs x \in T$, then $\lokpol (\bs x)$ as in \eqref{eq:locpol} is the solution to the weighted least squares problem
\begin{equation*}
	\lokpol(\bs x) =\argmin_{\bs \vartheta \in \R^{N_{m,d}} }\Big( -2\, \bs \vartheta^\top \bs a_{\bs p,h}(\bs x)+\bs \vartheta^\top B_{\bs p,h}(\bs x) \, \bs \vartheta \Big) \,.
\end{equation*}
The solution is determined by the normal equations
\begin{equation*}
	B_{\bs p,h}(\bs x) \, \lokpol(\bs x) =\bs a_{\bs p,h}(\bs x)\,,
\end{equation*}
where the matrix $B_{\bs p,h}(\bs x) $ is positive semidefinite. In particular, if $B_{\bs p,h}(\bs x)$ is positive definite, the solution in \eqref{eq:locpol} is uniquely determined and we obtain
\begin{equation*}
	\mulp(\bs x) =  \sum_{\bs j=\bs 1}^{\bs p} \wjb xh \, \oY{\bs j}
\end{equation*}
with
\begin{align} \label{eq:weights:locpol}
	\wjb xh \defeq \frac1{\bs p^{\bs 1}h^d} U_m^\top(\bs 0) \, B_{\bs p,h}^{-1}(\bs x) \, U_h(\bs{x_j}-\bs x)\,  K_h(\bs x_{\bs j}-\bs x)\,,
\end{align}
so that the local polynomial estimator is a linear estimator. In the following Lemma we show a result that implicates that $B_{\bs p, h}(\bs x)$ is positive definite.

\begin{lemma}\label{lemma:proof:(LP1)}
	Suppose that the kernel $K$ satisfies \eqref{eq:kernel} in Lemma \ref{lemma:locpol:weights}. Then there exist a sufficiently large $p_0 \in \N$ and a sufficiently small  $h_0 \in \R_+$ such that for all $\bs p$ with $\bs p_{\min} \geq p_0$ and $h\in(c/\bs p_{\min},h_0]$, where $c \in \R_+$ is a sufficiently large constant, the smallest eigenvalues $\lambda_{\min} \big(B_{\bs p, h}(\bs x)\big)$ of the matrices $B_{\bs p, h}(\bs x)$, which are given in \eqref{eq:Bp}, are bounded below by a universal positive constant $\lambda_0 >0$ for any $\bs x \in [0,1]^d$. 
\end{lemma}

An immediate consequence of Lemma \ref{lemma:proof:(LP1)} is the invertibility of $B_{\bs p, h}(\bs x)$ for all $ \bs p$ with $\bs p_{\min} \geq \bs p_0$, $h\in(c/\bs p_{\min},h_0]$ and $\bs x \in [0,1]^d$, and hence also the uniqueness of the local polynomial estimator for these $\bs p$ and $h$. In \citet[Lemma 1.5]{tsybakov2008introduction} the lower bound for the smallest eigenvalues has only be shown for a fixed sequence $h_p$ (for $d=1$) of bandwidths which satisfies $h_p \to 0$ and $p \, h_p \to \infty$. In contrast, we allow an uniformly choice of $h$ which results, in particular, in the findings of Section \ref{sec:estimation:rates}.

\begin{proof}[Proof of Lemma \ref{lemma:proof:(LP1)}] 
	In the following let $\bs v \in \R^{N_{m,d}}$. We show that there exist a sufficiently large $p_0 \in \N$ and a component wise sufficiently small $h_0\in\R_+$ such that the estimate 
	\begin{equation}\label{eq:inf:Bp}
		\inf_{\bs p\geq\bs  p_0} \inf_{h\in(c/\bs p_{\min}, h_0]} \inf_{\bs x \in [0,1]^d} \inf_{\normz{\bs v}=1} \bs v^\top B_{\bs p, h}(\bs x) \,\bs  v \geq \lambda_0 
	\end{equation}
	is satisfied. Then we obtain for these $\bs p$ and $h$, and any $\bs x \in [0,1]^d$ also
	\begin{align*}
		\lambda_{\min}\big(B_{\bs p, h}(\bs x)\big) &=\bs e_{\min}^\top \big(B_{\bs p, h}(\bs x)\big) B_{\bs p, h}(\bs x) \, \bs e_{\min}\big(B_{\bs p, h}(\bs x)\big) \geq \inf_{\normz{\bs v}=1} \bs v^\top B_{\bs p, h}(\bs x) \,\bs v \geq \lambda_0\,,
	\end{align*} 
	where $\bs e_{\min}\big(B_{\bs p, h}(\bs x)\big) \in \R^{m+1}$ is a normalized eigenvector of $\lambda_{\min}\big(B_{\bs p, h}(\bs x)\big)$. Let $\varPsi\colon \{0,1\}^d \mapsto \{1,\ldots,2^d\}$ be an enumeration of the subsets of $\{0,1\}^d$ with $\varPsi(\bs 0)=1$ and $E^i\defeq J_{(\varPsi^{-1}(i))_1} \times \cdots \times J_{(\varPsi^{-1}(i))_d}$ for all $i=1,\ldots,2^d$, whereby $J_0\defeq [0,\Delta)$ and $J_1\defeq (-\Delta,0]$ with the constants $\Delta\in \R_+$ and $K_{\min}>0$ given in \eqref{eq:kernel}. We set
	\begin{equation*}
		\lambda_i(\bs{v}) \defeq \fmin^{\bs 1} \, K_{\min} \int_{E^i} \big(\bs{v}^\top U_m(\bs z)\big)^2\dx \bs z \,, \quad \quad \lambda_i \defeq \inf_{\normz{\bs{v}} =1} \lambda_i(\bs{v})
	\end{equation*}
	for all $i=1,\ldots,2^d$ and $\fmin^{\bs1} \defeq \prod _{k=1}^d \fkmin$ given in Assumption \ref{ass:designdensity}. Applying \citet[Lemma 1.4]{tsybakov2008introduction} with $K(\bs z) = \one_{E^i}(\bs z)$ leads to
	\begin{equation*}
	\lambda_i(\bs{v})\geq \lambda_i >0\,.
	\end{equation*}
	Therefore we find a $\lambda_0>0$ such that $\min(\lambda_1,\ldots,\lambda_{2^d})>\lambda_0>0$, e.g. $\lambda_0\defeq \min(\lambda_1,\ldots,\lambda_{2^d})/2$. Now we want to specify  a partition $S_1\,\cup \ldots \cup \,S_{2^d}=[0,1]^d$ and functions $A^{(1)}_{p,h}(\bs x,\bs{v}),\ldots ,A^{(2^d)}_{p,h}(\bs x,\bs{v})$ such that 
	\begin{align}\label{eq:vBv:A}
	\bs{v}^\top B_{\bs p, h}(\bs x) \, \bs{v} \geq A^{(i)}_{\bs p,h}(\bs x,\bs{v})\,,\quad\norm{\bs v}_2 = 1, \quad \bs x \in S_i \,,\quad i =1,\ldots,2^d \,,
	\end{align} 
	and
	\begin{align}\label{eq:glm:vBv}
		\sup_{\bs x \in S_i} \sup_{\normz{\bs{v}} =1} \big|A^{(i)}_{\bs p,h}(\bs x,\bs{v}) - \lambda_i(\bs{v})\big| \leq \sum_{k=1}^d\frac {c_1}{p_k\,h}\,,\quad i =1,\ldots,2 
	\end{align}
	hold true for a positive constant $c_1>0$. Setting $c\defeq d\,c_1/(\min(\lambda_1,\ldots,\lambda_{2^d})-\lambda_0)$ yields $\sum_{k=1}^d c_1/(p_k\,h) < d\,c_1/c \leq \lambda_i-\lambda_0$ for all $i =1,\ldots,2^d$, since $h \in (c/\bs{p}_{\min}, h_0]$. Hence, it follows by \eqref{eq:vBv:A} and \eqref{eq:glm:vBv} that
		\begin{align*}
		\inf_{\bs{x} \in S_i} \inf_{\normz{\bs{v}} =1} \bs{v}^\top B_{\bs p, h}(\bs{x}) \, \bs{v}&=\inf_{\bs{x} \in S_i} \inf_{\normz{\bs{v}} =1}\big(\lambda_i(\bs{v})+\bs{v}^\top B_{\bs p, h}(\bs x) \,\bs{v}-\lambda_i(\bs{v})\big)\\
		&\geq\lambda_i+\inf_{\bs{x} \in S_i} \inf_{\normz{\bs{v}} =1} \big(A^{(i)}_{p,h}(\bs{x},\bs{v})-\lambda_i(\bs{v})\big)\\
		&=\lambda_i-\sup_{\bs{x} \in S_i} \sup_{\normz{\bs{v}} =1} \big(\lambda_i(\bs{v})-A^{(i)}_{p,h}(\bs{x},\bs{v})\big)\\
		&\geq\lambda_i-\sup_{\bs{x} \in S_i} \sup_{\normz{\bs{v}} =1} \big|A^{(i)}_{p,h}(\bs{x},\bs{v})-\lambda_i(\bs{v})\big|\\
		&\geq\lambda_i-\sum_{k=1}^d\frac {c_1}{p_k\,h} \geq \lambda_0\,,
	\end{align*}
	which leads to \eqref{eq:inf:Bp} because of
	\begin{align*}
		\inf_{\bs{p}\geq \bs{p}_0} \inf_{h\in(c/\bs p_{\min}, h_0]}& \inf_{\bs{x} \in [0,1]^d} \inf_{\normz{\bs{v}} =1} \bs{v}^\top B_{\bs p, h}(\bs x) \, \bs{v} \\
		&= \inf_{\bs{p}\geq \bs{p}_0}\inf_{h\in(c/\bs p_{\min}, h_0]}\min_{i=1,\ldots,2^d} \Big(\inf_{\bs{x} \in S_i} \inf_{\normz{\bs{v}} =1} \bs{v}^\top B_{\bs p, h}(\bs x) \, \bs{v} \Big) \geq \lambda_0 \,.
	\end{align*}

	\medskip

    Now we show \eqref{eq:vBv:A} and \eqref{eq:glm:vBv}.
	Let $I_0^k \defeq \big[0,1-\Delta h - 1/(2\fkmin \,p_k) \big]$ and $I_1^k \defeq \big[1-\Delta h - 1/(2\fkmin \,p_k),1 \big]$. We define $S_i\defeq I^1_{(\varPsi^{-1}(i))_1} \times \ldots \times I^d_{(\varPsi^{-1}(i))_{d}}$. It is clear that we get $2^d$ different subsets and $\bigcup_{i=1}^{2^d}S_i=[0,1]^d$. We indicate the design points on the $k^{\text{th}}$ axis by $x_{k,1},\ldots,x_{k,p_k}$ and define $y_{k,j}\defeq (x_{k,j}-x_k)/h$, $x_k \in [0,1]$ for all $1 \leq j \leq p_k$ and $1\leq k \leq d$ and let $\bs y_{\bs j}\defeq (y_{1,j_1},\ldots,y_{d,j_d}) \in \R^d$. We have to differentiate in two different cases. First let $x_k \in I_0^k$
	\begin{align*}
	y_{k,1} & \leq \frac{1}{2\fkmin  \, p_k\,h} - \frac{x_k}{h} \leq \frac{1}{2\fkmin  \, p_k\,h} \,, \\
	y_{k, p_k}&\geq \frac{1-\frac{1}{2\fkmin  \, p_k}-x_k}{h} = \frac{1-x_k}{h}-\frac1{2\fkmin\,p_k\,h} \\
	&\geq \frac{1-\big(1-\Delta h - (2\fkmin \,p_k)^{-1}\big)}{h} - \frac 1{2\fkmin  \, p_k\,h} = \Delta
	\end{align*} 
	by \eqref{lemma:design:points:auxiliary:1} in Lemma \ref{lemma:design:points:auxiliary}. For appropriate $\bs p$ and $h\in (c/\bs p_{\min}, h_0]$ the quantity $(2\fkmin \,p_k\,h)^{-1}$ gets small if $c$ is chosen large enough. Consequently, the points $y_{k,1},\ldots,y_{k, p_k}$ form a grid which covers an interval containing $\big[1/(2\fkmin \,p_k\,h), \Delta \big]$. 
	Hence, there exist $1 \leq j_{k*} \defeq j_{k*}(x_k) <  j_k^{*}\defeq j_k^{*}(x_k) \leq p_k$ such that $y_{k, j_{k*}}\geq 0 \; \land \; (j_{k*}=1\;\lor\; y_{k, j_{k*}-1}<0) $ and $y_{k, j_k^*} \leq \Delta \; \land \;  y_{k, j_k^*+1}>\Delta.$ 
	Here $\land$ denotes the logical and, and $\lor$ the logical or. For the second case let $x_k\in I_1^k$. Then we obtain the estimates 
	\begin{align*}
		y_{k,1}& \leq \frac1{2\fkmin p_k\,h}- \frac{x_k}{h} \leq \Delta -\frac1{h}\leq -\Delta,\\
		y_{k, p_k} & \geq -\frac{1}{2\fkmin p_k\,h}
	\end{align*}
	since $h$ has to be sufficiently small and $p_k$ sufficiently large. Consequently, the points $y_{k,1},\ldots,y_{k, p_k}$ form a grid which covers an interval containing $[-\Delta, -1/(2\fkmin p_k\,h)].$ Therefore we could define $\tilde{j_{k*}}$ and $\tilde{j_k^*}$ analogously as before but in the following we only examine the case for $S_1=I_0^{1}\times \ldots \times I_0^{d}$. The arguments for the other subsets are basically the same after the right case distinctions.
	We set $g(\bs x)\defeq(\bs{v}^\top U_m(\bs{x}) )^2$ to save some space in the upcoming calculations. By Assumption \ref{ass:designdensity}, inequality \eqref{lemma:design:points:auxiliary:2} in Lemma \ref{lemma:design:points:auxiliary} for the $y_{k,j}$ and the presentation of $B_{\bs p, h}(\bs x)$ in \eqref{eq:Bp} we get for $\bs{x} \in S_1$ 
	\begin{align*}
		\bs{v}^\top B_{\bs p, h}(\bs x) \, \bs{v}&= \frac1{\bs{p^1}h^d} \bs{v}^\top \bigg(\sum_{\bs j=\bs1}^{\bs p} U_m \big(\bs{y_j}\big) \, U_m^\top \big(\bs{y_j}\big) \, K\big(\bs y_{\bs j}(\bs{x})\big) \bigg) \bs{v} \nonumber\\
		&\geq \frac{\fmin^{\bs1}K_{\min}}{\fmin^{\bs1} \, \bs{p^1}h^d}\sum_{\bs j=\bs 1}^{\bs p} \Big(\bs{v}^\top U_m\big(\bs{y_j}\big) \Big)^2 \one_{[-\bs \Delta, \bs \Delta] } \big(\bs{y_j}\big) \nonumber\\
		&\geq \fmin^{\bs1}K_{\min} \sum_{j_1=1}^{p_1-1}\cdots\sum_{j_d=1}^{p_d-1} g(\bs{y_j}) (y_{1,j_1+1}-y_{1,j_1})\cdots(y_{d,j_d+1}-y_{d,j_d})\one_{[\bs 0, \bs \Delta)}(\bs y_{(j_1,\ldots,j_d)})\\
		&\geq \fmin^{\bs1}K_{\min} \sum_{j_{1}=j_{1*}}^{j_1^*}\cdots\sum_{j_{d}=j_{d*}}^{j_d^*} g(\bs{y_j}) (y_{1,j_1+1}-y_{1,j_1})\cdots(y_{d,j_d+1}-y_{d,j_d}) \defeql A^{(1)}_{\bs p,h}(\bs{x},\bs{v})\,.
	\end{align*}
	Inserting this function in \eqref{eq:glm:vBv}, dropping the scalar $\fmin^{\bs1}K_{\min}$ and oppressing the sups yields
\allowdisplaybreaks
	\begin{align}
		\bigg| \sum_{j_{1}=j_{1*}}^{j_1^*}\cdots\sum_{j_{d}=j_{d*}}^{j_d^*} & g(\bs{y_j}) (y_{1,j_1+1}-y_{1,j_1})\cdots(y_{d,j_d+1}-y_{d,j_d})-\int_{E^1} g(\bs z)\,\dx z_1 \cdots \dx z_d \bigg| \notag \\ 
		&\leq\sum_{k=1}^d \bigg| \int_0^{\Delta_1}\cdots \int_0^{y_{k,k_*}} \cdots \int_0^{\Delta_d} g(\bs z)\, \dx z_1 \cdots \dx z_d \bigg| \notag \\
		& \quad + \sum_{j_1=j_{1*}}^{j_1^*-1}\cdots \sum_{j_k=j_{k*}}^{j_k^*-1} \bigg| \int_{y_{1,j_1}}^{y_{1,j_1+1}}\cdots\int_{y_{d,j_d}}^{y_{d,j_d+1}} \big( g(\bs y_{\bs j}) - g(\bs z) \big) \,\dx z_1 \cdots \dx z_d \bigg| \label{eq:sup:sup} \\
		& \quad + \Delta^{d-1} \sum_{k=1}^d\bigg|g(\bs y_{(k_1^*,\ldots,k_d^*)}) \big(y_{k,j_{k}^*+1}-y_{k,j_k^*}\big) \bigg| \notag\\
		&\quad +\sum_{k=1}^d \bigg|\int_0^{\Delta_1} \cdots \int_{y_{k,j_k^*}}^{\Delta} \cdots \int_0^{\Delta_d} g(\bs z)\,\dx z_1 \cdots \dx z_d \bigg|\, . \notag 
	\end{align}
	We want to bound the order of these terms separately. Firstly, note that $g(\bs z)=\big(\bs{v}^\top U_m(\bs z) \big)^2$ for all coordinate wise bounded $\bs z$. Further $g(\bs z)$ is \name{Lipschitz}-continuous since it is a multivariate polynomial. Therefore it holds that $\abs{g(\bs{y_j}) - g(\bs z)} = \mc O\big((\bs p_{\min}h)^{-1}\big)$. Together with the fact that $\abs{y_{k, j_{k*}}}, \abs{ \Delta - y_{k,j_k^*}}, \abs{ y_{k,j_k+1} - y_{k,j_k} } = \mc O\big((p_k\,h)^{-1}\big)$ and $\sum_{j_k = j_{k*}}^{j_k^*-1} 1 = \mc O(p_k\,h)$ by Lemma \ref{lemma:design:points} we get that 
 \begin{align*}
     \eqref{eq:sup:sup} = \mc O\bigg( \frac{d}{\bs p_{\min}\,h} \bigg)\,.
 \end{align*}

\end{proof}

Now we can prove Lemma \ref{lemma:locpol:weights}.

\begin{proof}[Proof of Lemma \ref{lemma:locpol:weights}]
	By Lemma \ref{lemma:proof:(LP1)} there exist a sufficiently large $p_0 \in \N$ and a sufficiently small  $h_0\in \R_+$ such that for all $\bs p = (p_1,\ldots,p_d)^\top$ with $\bs p_{\min} \geq  \bs p_0$ and $h\in(c/\bs p_{\min}, h_0]$, where $c>0$ is a sufficiently large constant, the local polynomial estimator in \eqref{eq:locpolest} with any order $m \in \N$ is unique and a linear estimator with weights given in \eqref{eq:weights:locpol}. Further, the mentioned Lemma together with Lemma \ref{lemma:design:points}, \citet[Lemma 1.3]{tsybakov2008introduction} and \eqref{eq:kernel} imply \ref{ass:weights:vanish} and \ref{ass:weights:sup} of Assumption \ref{ass:weights} for the weights of the local polynomial estimator. Moreover, Lemma \ref{lemma:proof:(LP1)} and \citet[Proposition 1.12]{tsybakov2008introduction} lead to \ref{ass:weights:polynom} with $\gamma=m$. Note that the mentioned Lemma has to be generalized to the $d$-dimensional case. In the following we make use of 
	\begin{align} \label{eq:inverseBp:2norm}
	\normb{B_{\bs p,h}^{-1}(\bs x)}_{M,2} \leq \frac{1}{\lambda_0} 
	\end{align}
	for all $\bs p\geq \bs p_0$, $h\in(c/\bs p_{\min}, h_0]$ and $\bs x \in [0,1]^d$, where $\norm{M}_{M,2}$ is the spectral norm of a symmetric matrix $M \in \R^{N_{m,d} \times N_{m,d}}$ what is an immediate consequence of Lemma \ref{lemma:proof:(LP1)}.
	
	We continue with the \name{Lipschitz}-condition \ref{ass:weights:lipschitz}, and divide the proof in three cases with respect to the fact whether the weights vanish or not. In the following let $\bs 1 \leq \bs j \leq \bs p$ and $\bs x,\bs y \in [0,1]^d$.
		
	Let $\min \big(\norm{\bs x-\bs x_{\bs j} }_\infty,\normi{\bs y-\bs x_{\bs j}}\big) > h$, then by \ref{ass:weights:vanish} both weights $\wjb xh$ and $\wjb yh$ vanish, and hence \ref{ass:weights:lipschitz} is clear.  
		
	Let $\max \big(\norm{\bs x-\bs x_{\bs j}}_\infty,\norm{\bs y-\bs x_{\bs j}}_\infty\big) > h$. We assume $\norm{\bs y - \bs x_{\bs j}}_\infty>h$ and $\norm{\bs x - \bs x_{\bs j}}_\infty \leq h$ without loss of generality. Once again \ref{ass:weights:vanish} leads to $\wjb yh = 0 $, and hence
	the \name{Cauchy-Schwarz} inequality, $\normz{U_m(\bs 0)}=1$ and \eqref{eq:inverseBp:2norm} imply
	\begin{align*}
	\big| \wjb xh - \wjb yh \big| &= \frac{1}{\bs p^{\bs 1} h^d} \big| U_m^\top(\bs 0) \, B_{\bs p,h}^{-1}(\bs x) \, U_h(\bs x_{\bs j}-\bs x) \big| \,  \big| K_{h}(\bs x_{\bs j}-\bs x) \big| \\
	&\leq \frac1{\bs p^{\bs 1} h^d} \normb{U_m(\bs 0)}_2 \, \normb{B_{\bs p, h}^{-1}(x) \, U_h(\bs{x_j}- \bs x)}_2 \,  \absb{K_{h}(\bs{x_j}- \bs x)} \\
	&\leq \frac1{\bs p^{\bs 1} h^d}  \normb{B_{\bs p,h}^{-1}(\bs x)}_{M,2} \, \normb{U_h(\bs{x_j}- \bs x)}_2 \,  \big| K_{h}(\bs{x_j}- \bs x) \big| \\
	&\leq \frac{1}{\lambda_0  \, \bs p^{\bs 1} h^d} \big| K_{h}(\bs{x_j}- \bs x) \big|  \Bigg(\sum_{|\bs r|=0}^m \bigg(\frac{(\bs{x_j}- \bs x)^r}{\bs{h^r} \bs r!}\bigg)^2\Bigg)^{\frac12} \\
	&\leq \frac{c_1}{\lambda_0 \, \bs p^{\bs 1} h^d} \big| K_{h}(\bs{x_j}- \bs x) \big| 
	\end{align*}
	for a positive constant $c_1>0$. In the last step we used the fact that the sum can't get arbitrarily large, also for  component wise small $h$, because the kernel $K$ has compact support in $[-1,1]^d$. If $\norm{\bs x-\bs y}_\infty> h$, we use the upper bound of the kernel function in \eqref{eq:kernel} and obtain
		\begin{align*}
		\absb{\wjb xh - \wjb yh}  \leq \frac{c_1 \, K_{\max}}{\lambda_0 \, \bs p^{\bs 1} h^d} \,.
	\end{align*}
	Conversely, if $\norm{\bs x-\bs y}_\infty \leq h$, we add $K_{h}(\bs{x_j}-\bs y) = 0$ and estimate
	\begin{align*}
		\absb{\wjb xh - \wjb yh} \leq \frac{c_1}{\lambda_0 \, \bs p^{\bs 1} h^d} \big| K_{h}(\bs{x_j}- \bs x) - K_{h}(\bs{x_j}-\bs y)\big| \leq  \frac{c_1 \, L_K}{\lambda_0 \, \bs p^{\bs 1} h^d} \, \cdot\,\frac{\norm{\bs x-\bs y}_\infty}h \, ,
	\end{align*}
	because of the \name{Lipschitz}-continuity of the kernel. So in total we get
	\begin{align} \label{proof:locpol:weights:1}
		\Big| \wjb xh - \wjb yh \Big| \leq \frac{c_1 \, \max(K_{\max},L_K)}{\lambda_0 \,\bs p^{\bs 1} h^d} \bigg(\frac{\norm{\bs x-\bs y}_\infty}h \wedge 1 \bigg) \,.
	\end{align}

	Let $\max \big(\norm{\bs x-\bs x_{\bs j}}_\infty,\norm{\bs y-\bs x_{\bs j}}_\infty\big) \leq h$, then both weights don't vanish and we have to show a proper \name{Lipschitz}-property for
	\begin{align*} 
		\wjb xh = \frac1{\bs p^{\bs 1} h^d} U_m^\top(\bs 0) \, B_{\bs p,h}^{-1}(\bs x) \, U_h(\bs{x_j}- \bs x)\,  K_{h}(\bs{x_j}- \bs x)\,.
	\end{align*}
	The kernel $K$ and polynomials on compact intervals are \name{Lipschitz}-continuous, hence it suffices to show that $B_{\bs p, h}^{-1}(x)$ has this property as well. Then the weights are products of bounded \name{Lipschitz}-continuous functions and, thus, also \name{Lipschitz}-continuous. The entries of the matrix $B_{\bs p, h}(\bs x)$, which is defined in \eqref{eq:Bp}, considered as functions from $[0,1]$ to $\R$ are \name{Lipschitz}-continuous. Indeed they are of order one by using Assumption \ref{ass:design1}.  Hence, the row sum norm $\normiM{B_{\bs p, h}(\bs x)-B_{\bs p, h}(\bs y)}$ is a sum of these \name{Lipschitz}-continuous functions and, in consequence, there exists a positive constant $L_\infty >0$ such that
	\begin{align*}
	\normiM{B_{\bs p, h}(\bs x)-B_{\bs p, h}(\bs y)} \leq (N_{m,d}+1) L_\infty \, \,\frac{\norm{\bs x-\bs y}_\infty}h\,
	\end{align*}
	is satisfied. Since the matrices are symmetric the column sum norm is equal to the row sum norm and we obtain
	\begin{align*}
		\normb{B_{\bs p, h}(\bs x)-B_{\bs p, h}(\bs y)}_{M,2} \leq \normb{B_{\bs p, h}(\bs x)-B_{\bs p, h}(\bs y)}_{M,\infty}\leq (N_{m,d}+1) L_\infty  \, \frac{\norm{\bs x-\bs y}_\infty}h \,.
	\end{align*}
	This leads together with \eqref{eq:inverseBp:2norm} and the submultiplicativity of the spectral norm to
	\begin{align*}
		\normb{ B_{\bs p, h}^{-1}(x)-B_{\bs p, h}^{-1}(y) }_{M,2} &= \normb{ B_{\bs p, h}^{-1}(y) \big( B_{\bs p, h}(\bs y)-B_{\bs p, h}(\bs x) \big) B_{\bs p, h}^{-1}(x)}_{M,2} \\
		&\leq \normb{ B_{\bs p, h}^{-1}(y)}_{M,2} \, \normb{ B_{\bs p, h}(\bs y)-B_{\bs p, h}(\bs x)}_{M,2} \, \normb{ B_{\bs p, h}^{-1}(x)}_{M,2} \\
		&\leq \frac{(N_{m,d}+1) L_\infty }{\lambda_0^2} \, \,\frac{\norm{\bs x-\bs y}_\infty}h\, \,,
	\end{align*}
	which is the \name{Lipschitz}-continuity of $B_{\bs p, h}^{-1}(x)$ with respect to the spectral norm. So finally there exists a positive constant $c_2>0$ such that 
	\begin{align*}
    	 \absb{\wjb xh - \wjb yh}  \leq  \frac{c_2}{2 \bs p^{\bs 1} h^d} \, \,\frac{\norm{\bs x-\bs y}_\infty}h \leq \frac{c_2}{\bs p^{\bs 1} h^d}
	\end{align*}
	is satisfied. Here we used in the last step that $\max \big(\norm{\bs x-\bs x_{\bs j}}_\infty,\normi{\bs y-\bs x_{\bs j}}\big) \leq h$ implies $\norm{\bs x-\bs y}_\infty \leq2\,h$.
	
Finally we choose $C_3 \geq \max\big( c_1 \max(K_{\max},L_K)/\lambda_0, c_2\big)$ and obtain with \eqref{proof:locpol:weights:1} finally Assumption \ref{ass:weights:lipschitz}. 
\end{proof}

\section{Proof of Lemma \ref{theorem:estimation:rates1}}
\label{sec:proof:bias}

\begin{proof}[Proof of Lemma \ref{theorem:estimation:rates1} i).]
	
	For the first part we want to use Taylor expansion and that the weights of the estimator reproduce polynomials of a certain degree. For certain $\theta_{\bs j} \in [0,1], \bs 1 \leq \bs j \leq \bs p,$ s.t. $\bs{\tau_j} \defeq \bs x + \theta_{\bs j}(\bs{x_j} - \bs x) \in [0,1]^d$  it holds
	\allowdisplaybreaks
	\begin{align*}
		\sum_{\bs j = \bs 1}^{\bs p} \wjb{x}{h}& \mu(\bs x_{\bs j}) - \mu (\bs x) = \sum_{\bs j = \bs 1}^{\bs p} \wjb{x}{h}( \mu(\bs x_{\bs j}) - \mu (\bs x))\\
		&=\sum_{\bs j = \bs 1}^{\bs p} \Big(\sum_{|\bs k|= 1}^{\gamma-1}D^{\bs k} \mu(\bs x) \frac{(\bs x_{\bs j}-\bs x)^{\bs k}}{\bs k !} +  \sum_{|\bs k|= \gamma}D^{\bs k}\mu(\bs \tau_{\bs j})\frac{(\bs x_{\bs j}-\bs x)^{\bs k}}{\bs k !}\Big) \wjb{x}{h}\\
		&=\sum_{\bs j = \bs 1}^{\bs p} \sum_{|\bs k|= \gamma}D^{\bs k}\mu(\bs \tau_{\bs j})\frac{(\bs x_{\bs j}-\bs x)^{\bs k}}{\bs k !} \wjb{x}{h}\tag{Taylor and \ref{ass:weights:polynom}}\\
		&=\sum_{\bs j = \bs 1}^{\bs p}\Bigg( \sum_{|\bs k|= \gamma}\bigg(D^{\bs k}\mu(\bs \tau_{\bs j})-D^{\bs k}\mu(\bs x)\bigg)\frac{(\bs x_{\bs j}-\bs x)^{\bs k}}{\bs k !}\Bigg) \wjb{x}{h}\tag{by \ref{ass:weights:polynom}}\\
		& \leq \sum_{\bs j = \bs 1}^{\bs p}\Bigg( \sum_{|\bs k|= \gamma}\norm{\theta(\bs{x_j} - \bs x)}_\infty^{\alpha-\gamma}\frac{(\bs x_{\bs j}-\bs x)^{\bs k}}{\bs k !}\Bigg) \wjb{x}{h}\\
		&\leq  \sum_{\bs j = \bs 1}^{\bs p}\norm{\bs x_{\bs j}- \bs x}_\infty^\alpha\Bigg( \sum_{|\bs k|= \gamma}\frac{1}{\bs k !}\Bigg) \wjb{x}{h}\\
		&\leq  \norm{\bs h}_\infty^\alpha \sum_{|\bs k|= \gamma}\frac{C_1}{\bs k !}.
	\end{align*}
\end{proof}

\newpage
\section{Additional numerical results}

\begin{figure}[h!]
	\centering
	\includegraphics[width=\linewidth]{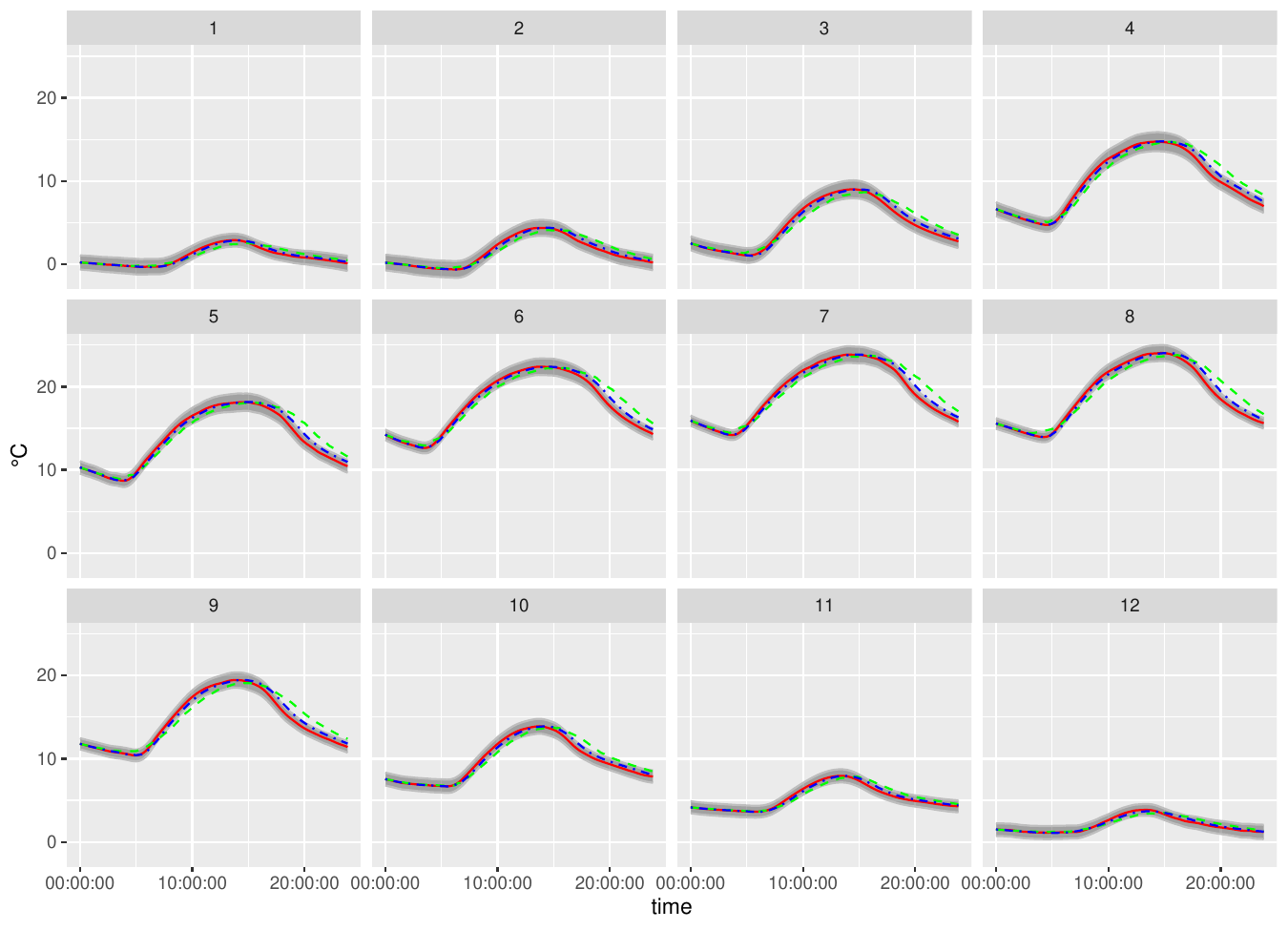}
	\caption{\small LP estimator with $p=144$ (10-minute data, red line), $p=24$ (one houre data, blue dotted/dashed line) and $p = 12$ (tow hour data, green dashed line). Confidence bands for the $p=144$ curve from \cite{liebl2019fast}.}
	\label{fig:weather_conf_bands}
\end{figure}

\section{Gaussian approximation for non sub-Gaussian errors} \label{sec:app:gaussian_approximation}

In the following we determine the stochastic rate of convergence for the error term $I_2^{n,p,h}$ for errors which only satisfy mild moment assumptions. We restrict ourselves to one-dimensions: $d=1$. To his end we consider the situation of Assumption \ref{ass:weights}. For the Gaussian approximation we need a further property for the weights. Let $\Cink>0$ be a constant not depending on $n,p$ and $h$. Then for $x\in [0,1]$ and $j = 1, \ldots, p-1,$
\begin{enumerate}[label=\normalfont{(W\arabic*)},leftmargin=9.9mm]
    \setcounter{enumi}{5}
    \item\label{ass:weights:increments} $\displaystyle \abs{\wj xh - w_{j+1}(x;h)} \leq \frac{\Cink}{(p\,h)^2}\ind_{ [x_j - h, x_j+h] \cup [x_{j+1}-h, x_{j+1}+h] }(x)\,. $ 
\end{enumerate}
In Lemma \ref{lem:weight:increments} we show that in one dimensions that the weights of the local polynomial estimator satisfy this assumption. Further we replace Assumption \ref{ass:model} as follows. 

\begin{assumption}\label{ass:non-subGaussian-Errors}
    The random variables $\{\e_{i,\bs j} \mid 1 \leq i \leq n,\, 1 \le j \le  p\}$ are centered, independent and independent of the processes $Z_1,\dotsc,Z_n$. Further we assume that for a $q\geq 4$ it holds $\sup_{i = 1, \ldots, n, \, j = 1,\ldots, p}\expec[\abs{\varepsilon_{i,j}}^q]\leq C <\infty$ for a constant $C>0$.
\end{assumption}

\begin{proposition}\label{prop:gaussian_approximation}
    In model \eqref{eq:model} with $d=1$, under Assumption $\ref{ass:design1}$ let the errors $\epsilon_{i,j}$ satisfy Assumption \ref{ass:non-subGaussian-Errors}. Then 
    \begin{align}
        I^{n,p,h}(x) & = \sum_{j = 1}^p \wj xh \frac1n\sum_{i = 1}^n \epsilon_{i,j}\,, \quad x \in [0,1]\,, 
    \end{align}
    with weights $\wj xh$ satisfying Assumption \ref{ass:weights} and \ref{ass:weights:increments} satisfies
    \begin{align}
        \normb{I^{n,p,h}}_\infty & = \mc O_\prob \Bigg( \sqrt{\frac{\log(1/h)}{n\,p\,h}} + \frac{p^{1/q}}{\sqrt n\, p\,h}\Bigg)\,.
    \end{align}
\end{proposition}

\begin{remark}
    Comparing the two rates yields that rate of the Gaussian approximation is slower then the initial rate if $h \lesssim p^{2/q}/(p\,\log(1/h)) \sim  p^{2/q}/(p\,\log(p)) $ due to 
    \begin{align*}
        \frac{p^{1/q}}{\sqrt{n}\,p\,h } \gtrsim \frac{p^{1/q}}{\sqrt{n\,p\,h}}\cdot \frac{\sqrt{\log(1/h)}}{p^{1/q}} \simeq \sqrt{\frac{\log(1/h)}{n\,p\,h}}\,
    \end{align*}
    and vice versa. Therefore we need to compare the optimal choice for the rate of convergence $h^\star$ in \eqref{eq:optimal:bandwidth} with $p^{2/q}/(p\,\log(1/p))$. This results in 
    \begin{align*}
         &  p \gtrsim \bigg( \frac{n}{\log (n)}\bigg)^{\frac{q}{2(\alpha(q-2)+1)}}.
    \end{align*}
    If this condition on $p$ is satisfied the $\sqrt n$-rate and the asymptotic normality are still guaranteed. For the worst case covered by the theory, $ q= 4$, this threshold for $p$ is larger then $n^{1/(2\alpha)}\log(n)$ and therefore the $\sqrt n$-rate can only be obtained for ratewise larger $p$ then before. However for large $q$ we have  
    \begin{align*}
        \bigg( \frac{n}{\log (n)}\bigg)^{\frac{q}{2(\alpha(q-2)-1)}} \stackrel{q\to \infty}{\;\longrightarrow\;} \bigg( \frac n {\log(n)}\bigg)^{\frac1{2\alpha}} \,,
    \end{align*}
    and in consequence the initial thresholds from Remark \ref{rem:regimes_in_the_rate} remain valid.
    

\begin{table}[t!]
	\begin{minipage}[b]{.45\linewidth}
        \centering
        \begin{tabular}{c|rrrr}
          \hline
         {\tiny \diagbox np} & $5$ & $10$ & $25$ & $50$ \\ 
          \hline
        $1$ & 1.011 & 1.018 & 0.999 & 1.024 \\ 
          $5$ & 1.000 & 1.022 & 1.006 & 1.020 \\ 
          $10$ & 1.010 & 1.007 & 1.005 & 0.996 \\ 
          $25$ & 0.999 & 1.013 & 1.014 & 1.009 \\ 
          $50$ & 1.002 & 1.013 & 1.001 & 0.993 \\ 
           \hline
        \end{tabular}
        \caption{\small Quotient of $\norm{\sum_{j = 1}^p  \wj \cdot h \bar \epsilon_{j,n}}_\infty$ with $\bar \epsilon_{j,n} \sim \mc N(0,n^{-1})$ divided by $\norm{\sum_{j = 1}^p  \wj \cdot h \frac1n\sum_{i=1}^n t_{i,j}}_\infty$ with $t_{i,j} \sim \sqrt{3/5}\, t_5$.}
        \label{tab:quot_I_error}
	\end{minipage}
\hspace{0.5cm}
	\begin{minipage}[b]{.45\linewidth}
        \centering
        \begin{tabular}{c|rrrr}
          \hline
        {\tiny \diagbox{n}{p} }& $5$ & $10$ & $25$ & $50$ \\ 
          \hline
        $1$& 0.975 & 0.929 & 0.858 & 0.799 \\ 
          $5$ & 0.992 & 0.972 & 0.933 & 0.904 \\ 
          $10$ & 0.998 & 0.979 & 0.962 & 0.946 \\ 
          $25$ & 0.993 & 0.996 & 0.979 & 0.973 \\ 
          $50$ & 0.995 & 1.002 & 0.990 & 0.980 \\ 
           \hline
        \end{tabular}
        \caption{\small Quotient of $\max_{j}\abs{\bar \epsilon_{j,n}}$ divided by $ \max_{j}\abs{\frac1n\sum_{i=1}^n t_{i,j}}$ with $\bar \epsilon_{j,n} \sim \mc N(0,n^{-1})$ and $t_{i,j} \sim \sqrt{3/5} \,t_5$.\\ [0.7mm]}
        \label{tab:quot_mean_error}
	\end{minipage}
\end{table}
    
Table \ref{tab:quot_I_error} contains the quotient of the error term $I_2^{n,p,h}$ with normal (nominator) and standardized $t_5$ distributed errors (denominator) for different $n, p\in \N$. Even for very small sample sizes, the smoothing already suffices for both terms to behave very similarly. Differences can however be seen without smoothing, which gives a further argument against the use of interpolation estimators. 

     
\end{remark}

\newpage

\begin{lemma} \label{lem:gaussian_appr}
    In model \eqref{eq:model} under Assumption $\ref{ass:design1}$ 
    let $X_1, \ldots, X_p$ be independent random variables which satisfy $ \expec[X_1] = 0 $, $\expec[X_1^2] = 1$ and $\expec [\abs{X_j}^q] \leq C < \infty$, as required in Assumption \ref{ass:non-subGaussian-Errors}. Consider the term 
    \begin{align}
        E^{p,h} (x) & = \sum_{j = 1}^p \wj xh X_j\,, \quad x \in [0,1]\,, 
    \end{align}
    where the weights $\wj xh$ satisfy Assumption \ref{ass:weights} and \ref{ass:weights:increments}. This term becomes small with the rate
    \begin{align}
        \normb{E^{p,h}}_\infty & = \mc O_\prob \Bigg( \sqrt{\frac{\log(1/h)}{p\,h}} + \frac{p^{1/q}}{p\,h}\Bigg)\,.
    \end{align}
\end{lemma}

    \begin{proof}[Proof of Lemma \ref{lem:gaussian_appr}]
    Consider $T>0$ and $q\geq 4$. Then from \citet[Corollary 4, §5]{sakhanenko1991accuracy} there exists a Brownian motion $(B_t)_{t \in [0,\infty)}$ such that 
    \begin{align*}
        \prob\Big( \max_{k = 1,\ldots, p }\absb{\sum_{j = 1}^k X_{j} - B_k} > q\,T\,p^{1/q}\Big) & 
        \leq \big( p\,T^{q}\big)^{-1} \sum_{j = 1}^p  \expec\big[ \abs{X_j}^q\big] + \prob \Big( \max_{j = 1,\ldots, p} \abs{X_j} > T\, q^{1/q}\Big) \\
        & \leq \frac{C}{T^q} + \prob\Big( \max_{j = 1,\ldots, p} \abs{X_j} > T\,p^{1/q} \Big)\,.
    \end{align*}
    Now, by using the independence and the existence of the fourth moment of the errors, we get
    \begin{align*}
        \prob\Big( \max_{j = 1,\ldots, p}\abs{X_j} > T\,p^{1/q} \Big) & = 1 - \prod_{j = 1}^p \prob\big(  \abs{X_j }\leq T \,p^{1/q} \big) \\
        & \leq 1 - \prod_{j = 1}^p \bigg( 1 - \frac{\expec[\abs{X_j}^q]}{p\,T^{q}}\bigg) \\
        & \leq 1 - \bigg( 1  - \frac{C\,T^{-q}}{p}\bigg) ^p \stackrel{p \to \infty}\to 1 - \exp\big( -C\,T^{-q}\big)\,.
    \end{align*}
    Therefore we get 
    \begin{align}
        \max_{k = 1,\ldots, p} \abss{ \sum_{j = 1}^k X_j - B_k } = \mc O_\prob\big( p^{1/q}\big)\,. \label{eq:partial_sum_approx}
    \end{align}
    Now set $N_j \defeq B_j - B_{j-1},j = 1,\ldots, p$ with $B_0 = 0$ by writing $\eta_{j,n} \defeq X_j - N_j$ we have $\eta_{j,n} = \sum_{ i = 1}^j \eta_{i,n} - \sum_{i = 1}^{j-1} \eta_{i,n}$. By setting $\sum_{i = 1}^0 \eta_{i,n}= 0$ we get
    \begin{align}
		\Big|\sum_{j=1}^p \eta_j\Big|&= \bigg|\sum_{j=1}^p X_j - B_p\bigg|\,.\label{eqn:inequality_eta}
	\end{align} 
	With the properties of the weights \ref{ass:weights:sup} and \ref{ass:weights:increments} and the previous statement we get
    \allowdisplaybreaks
	\begin{align*}
		\sup\limits_{x \in [0,1]}\Big|E^{p,h}(x)&- \sum_{j = 1}^p \wj xh  N_j\Big|\\
		&=\sup\limits_{x \in [0,1]}\Big|w_{p}(x;h)\sum_{j=1}^p\eta_j + \sum_{j=1}^{p-1}\big(\wj{x}{h}-w_{j+1}(x;h)\big)\sum_{i=1}^j\eta_i\Big|\\
		&\leq \sup\limits_{x \in [0,1]} \max\limits_{j=1,\ldots,p}|\wj{x}{h}| \Big|\sum_{j=1}^p \eta_j\Big|\\
		&\qquad+\max\limits_{j=1,\ldots,p}\Big| \sum_{i=1}^j\eta_i\Big| \sup\limits_{x \in [0,1]} \Big|\sum_{j=1}^{p-1}\big( \wj{x}{h}-w_{j+1}(x;h)\big)\Big|\\
		&\leq \sup\limits_{x \in [0,1]} \frac{\Cmax}{ph} \bigg|\sum_{j=1}^p  X_j - B(p)\bigg|\tag{with \ref{ass:weights:sup} and \eqref{eqn:inequality_eta}}\\
		&\qquad+\max\limits_{j=1,\ldots,p} \bigg|\sum_{i=1}^j  X_i - B(j)\bigg| \sup\limits_{x \in [0,1]} \Big|\sum_{j=1}^{p-1}\big( \wj{x}{h}-w_{j+1}(x;h)\big)\Big|\\
		&\leq \bigg( \frac{\Cmax}{ph}  + \frac{\Ccard\,\Cink}{p\,h} \bigg) \cdot \mc O_\prob\big( p^{1/q}\big)\tag{with Assumption \ref{ass:design1}, \eqref{eq:partial_sum_approx}, \ref{ass:weights:vanish} and \ref{ass:weights:increments}}\\
        & = \mc O_\prob  \bigg( \frac{p^{1/q}}{p\,h}\bigg) \,.
	\end{align*}
    Applying  Lemma \ref{theorem:estimation:rates1} ii) to the term $\sum_{j = 1}^p \wj xh N_j$ for the case $n = 1$ and $d = 1$ we get that the rate of this term is bounded by $\sqrt{\log(1/h) / (p\,h)}$ which concludes the proof. 
\end{proof}

\begin{proof}[Proof of Proposition \ref{prop:gaussian_approximation}]
    Since $\sigma \geq \sigma_j$ for all $j = 1, \ldots, p$, we have 
    \begin{align*}
        \normb{I^{n,p,h}}_\infty & \leq \frac\sigma{\sqrt n}\,\sup_{x \in [0,1]} \abss{\sum_{j = 1}^p \wj xh \frac{1}{\sigma_j\,\sqrt n} \sum_{ i = 1}^n \epsilon_{i,j}}\,.
    \end{align*}
    Setting $\tilde \epsilon_{j} = (\sigma_j\,\sqrt{n})^{-1} \sum_{i = 1}^n \epsilon_{i,j}$ we have that $\expec[\tilde \epsilon_j] = 0$ and $\expec[\tilde \epsilon_j^2] = 1$ by by  Assumption \ref{ass:non-subGaussian-Errors}. By using the Rosenthal inequality \citep{rosenthal1970subspaces} we further get for independent, centered random variables 
    $\expec\big[ \abs{\tilde \epsilon_{j}}^q\big] \leq K_q\, \max\big\{ C\,n^{1-q/2}, 1 \big\}$, where $K_q>0$ is a constant only depending on $q$. Now $\tilde \epsilon_j$ suffice Assumption \ref{ass:non-subGaussian-Errors} and therefore the claim follows from Lemma \ref{lem:gaussian_appr}.
\end{proof}

\begin{lemma}\label{lem:weight:increments}
    Consider the same setup as in Lemma \ref{lemma:locpol:weights} in the one dimensional case. Then the weights the local polynomial estimator also suffice \ref{ass:weights:increments}.
\end{lemma}

\begin{proof}
    With the upper bound for $\norm{B^{-1}_{p, h}(x)}_{M,2} \leq \lambda_0^{-1}$ from \eqref{eq:inverseBp:2norm} and the Lipschitz continuity of $U_m(x)K( x)$ we get
	\begin{align*}
		\big|\wj xh-w_{j+1}(x, h)\big|
		&\leq \frac{1}{ph}\normb{B_{p,h}^{-1}(x)\big(U_h\big(x_j-x\big)K_h\big(x_j-x\big)-U_h\big(x_{j+1}-x\big)K_h\big(x_{j+1}-x\big)\big)}_2\\
		&\leq \frac{L}{\lambda_0 \,p\,h} 		\frac{|x_j-x_{j+1}|}{h}\,\ind_{[x_j-h,x_j+h]\cup [x_{j+1}-h,x_{j+1}+h]}(x)\\
        & \leq \frac{L}{\lambda_0\,f_{\min} \,p^2\,h^2}\,\ind_{[x_j-h,x_j+h]\cup [x_{j+1}-h,x_{j+1}+h]}(x)\,,
	\end{align*}
	due to the Cauchy-Schwarz inequality, $\norm{U(0)}_2=1$ and \eqref{lemma:design:points:auxiliary:2}. Setting $\Cink = L/(\lambda_0\,f_{\min})$ yields the claim.
\end{proof}

\end{document}